\documentclass[a4paper]{amsart}
\usepackage[latin1]{inputenc}
\usepackage{amssymb}
\usepackage{amsmath}
\usepackage{mathtools}
\usepackage{mathrsfs}
\usepackage{eufrak}
\usepackage{amsthm}
\usepackage{amsfonts}
\usepackage{textcomp}
\usepackage{graphicx}
\usepackage[pdftex]{color}
\usepackage{paralist}
\usepackage[shortlabels]{enumitem}
\usepackage{hyperref}
\usepackage{comment}
\usepackage[arrow, matrix, curve]{xy}
\usepackage{tikz} 
\include{quiver.sty}
\usetikzlibrary{matrix,patterns,shapes,decorations.pathmorphing,decorations.pathreplacing,calc,arrows} \usepackage{tikz-cd} 
\usepackage{tkz-graph}
\usepackage{mdwlist}

\usepackage[capitalise, noabbrev]{cleveref}%references, include last
\newtheorem*{corollary*}{Corollary}
\newtheorem{theorem}{Theorem}[section]

\newtheorem{corollary}[theorem]{Corollary}
\newtheorem{lemma}[theorem]{Lemma}
\newtheorem{proposition}[theorem]{Proposition}
\newtheorem{question}[theorem]{Question}

\newtheorem*{claim*}{Claim}

\theoremstyle{definition}
\newtheorem{definition}[theorem]{Definition}

\newtheorem*{theorem }{Theorem}
\newtheorem{remark}[theorem]{Remark}

\newtheorem{problem}[theorem]{Problem}
\newtheorem{example}[theorem]{Example}

\theoremstyle{remark}

\numberwithin{equation}{theorem}

\makeatletter
\renewcommand{\mod}{\operatorname{mod}}
\newcommand{\proj}{\operatorname{proj}}
\newcommand{\ind}{\operatorname{ind}}
\newcommand{\inj}{\operatorname{inj}}

\newcommand{\Ext}{\operatorname{Ext}}
\newcommand{\Tor}{\operatorname{Tor}}
\newcommand{\End}{\operatorname{End}}
\newcommand{\Hom}{\operatorname{Hom}}
\newcommand{\add}{\operatorname{\mathrm{add}}}

\newcommand{\rad}{\operatorname{\mathrm{rad}}}
\newcommand{\soc}{\operatorname{\mathrm{soc}}}

\newcommand{\AR}{\operatorname{AR}}
\newcommand{\Tr}{\operatorname{Tr}}
\newcommand{\Mod}{\operatorname{Mod}}
\newcommand{\refl}{\operatorname{ref}}
\newcommand{\grade}{\operatorname{grade}}
\newcommand{\sgrade}{\operatorname{s.grade}}

\renewcommand{\mod}{\operatorname{mod}}

\newcommand{\Z}{\mathbb{Z}}
\newcommand{\A}{\mathcal{A}}
\newcommand{\D}{\mathcal{D}}
\newcommand{\E}{\mathcal{E}}
\newcommand{\RHom}{\operatorname{\mathsf{R}Hom}}
\newcommand{\REnd}{\operatorname{\mathsf{R}End}}
\DeclareMathOperator{\cHom}{\mathcal{H}\hspace{-1pt}{\it om}}

\renewcommand{\ker}{\mathrm{Ker}}
\newcommand{\Ker}{\operatorname{Ker}}
\newcommand{\coker}{\mathrm{Coker}}
\newcommand{\Coker}{\operatorname{Coker}}
\renewcommand{\Im}{\operatorname{Im}}
\newcommand{\Db}{\mathrm{D}^{\mathrm{b}}}

\newcommand{\idim}{\operatorname{idim}}
\newcommand{\pdim}{\operatorname{pdim}}
\newcommand{\gldim}{\operatorname{gldim}}

\newcommand{\Fac}{\mathrm{Fac}}

\newcommand{\Y}{\mathcal{Y}}
\def\T{\mathcal{T}}
\def\F{\mathcal{F}}
\def\C{\mathcal{C}}
\def\op{\mathrm{op}}
\def\La{\Lambda}
\def\ij{\mathrm{i}}
\def\ninj{\mathrm{ni}}
\def\rD{\mathrm{D}}

\def\Db{\mathrm{D^b}}

\def\rsimeq{\rotatebox{-90}{$\simeq$}}
\def\xsimeq{\xrightarrow{\simeq}}
\def\ysimeq{\xleftarrow{\simeq}}
\def\Om{\Omega}
\def\del{\delta}
\def\lotimes{\otimes^\mathrm{L}}
\def\disoplus{\displaystyle\bigoplus}
\def\Ga{\Gamma}
\def\Si{\Sigma}
\begin{document}

\title[Spherical modules and the Auslander--Gorenstein condition]{Spherical modules and the Auslander--Gorenstein condition for Auslander--Yoneda algebras}

\date{\today}

\begin{abstract}
For a finite dimensional algebra $A$ of finite global dimension we study the Auslander--Yoneda algebra defined as the Yoneda algebra of the direct sum of all indecomposable $A$-modules. We show that the Auslander--Yoneda algebra is an Auslander--Gorenstein algebra if and only if every indecomposable left and right $A$-module is spherical in the sense of Auslander and Bridger. 
This motivates the study of spherical algebras defined by the condition that every indecomposable module is spherical.
We characterize spherical algebras by a certain natural pair of subcategories being a split torsion pair.
Moreover, we prove that representation-finite algebras which are spherical are directed and give a full classification of spherical Nakayama algebras. Furthermore, we show that replicated algebras of hereditary algebras are spherical. 
As a final application of the new notion of spherical algebras, we give a negative answer to a question of Venjakob on Auslander regular algebras in general, but show that there is a positive answer when assuming that every indecomposable left $A$-module is spherical.
\end{abstract}



\author[F. Bauwens]{Finn Bauwens}
\address[Finn Bauwens]{Mathematical Institute of the University of Bonn, Endenicher Allee 60, 53115 Bonn, Germany}
\email{s6fibauw@uni-bonn.de}

\author[N. Hanihara]{Norihiro Hanihara}
\address[Norihiro Hanihara]{Faculty of Mathematics, Kyushu University, 744 Motooka, Nishi-ku, Fukuoka, 819-0395, Japan}
\email{hanihara@math.kyushu-u.ac.jp}

\author[R. Marczinzik]{Ren\'{e} Marczinzik}
\address[Ren\'{e} Marczinzik]{Mathematical Institute of the University of Bonn, Endenicher Allee 60, 53115 Bonn, Germany}
\email{marczire@math.uni-bonn.de}

\author[M.~Plogmann]{Marvin Plogmann}%
\address[Marvin Plogmann]{%
	Mathematisches Institut, %
	Universit\"at zu K\"oln, %
	Weyertal 86-90, %
	50931 K\"oln, %
	Germany%
}%
\email{plogmann@math.uni-koeln.de}%
\urladdr{https://sites.google.com/view/marvinplogmann}%

\author[J.~Thomm]{Jan Thomm}%
\address[Jan Thomm]{%
	Mathematisches Institut, %
	Universit\"at zu K\"oln, %
	Weyertal 86-90, %
	50931 K\"oln, %
	Germany%
}%
\email{jthomm@math.uni-koeln.de}%

%\thanks{This research started at the Junior Research Retreat in Bad Boll 2025 which was supported by the Hausdorff School for Mathematics (HSM) and funded by the Deutsche Forschungsgemeinschaft (DFG, German Research Foundation) under Germany's Excellence Strategy -- EXC-2047/2 -- 390685813.
%The second author is supported by JSPS KAKENHI Grant Numbers JP22KJ0737/JP25K17233. The second author is supported by Kavli Institute for the Physics and Mathematics of the Universe (WPI), The University of Tokyo, as an affiliate member.}

\maketitle

\setcounter{tocdepth}{1}
\tableofcontents

\section{Introduction}

A two-sided noetherian ring $\La$ is called \emph{Auslander--Gorenstein}
if there is a finite injective coresolution 
$$0 \rightarrow \La \rightarrow I^0 \rightarrow I^1 \rightarrow \cdots \rightarrow I^n \rightarrow 0$$
of the regular module $\La$ such that the flat dimension of $I^i$ is at most $i$ for all $i \geq 0$.
Auslander--Gorenstein rings are the non-commutative generalisation of the classical local commutative Gorenstein rings in algebraic geometry by results of Bass \cite{B}. Several other important classes of rings are Auslander--Gorenstein, such as enveloping algebras of finite dimensional Lie algebras, Weyl algebras and more generally the ring of $\mathbb{C}$-linear differential operators on an irreducible smooth subvariety of affine
space, we refer for example to the textbook \cite{VO} and the two surveys \cite{Cl} and \cite{KM} for more on Auslander--Gorenstein algebras.
An important feature of Auslander--Gorenstein rings $\La$ that are finitely generated over a noetherian center is that they come with a canonical bijection acting on the simple modules, the \emph{grade bijection}, first discovered by Iyama in \cite{I1}.
When $\La$ is even a finite dimensional algebra, Iyama's grade bijection coincides with another bijection for Auslander--Gorenstein algebras discovered by Auslander and Reiten \cite{AR1} as was recently shown in \cite{KMT}.
This \emph{Auslander--Reiten} bijection for Auslander--Gorenstein algebras sends indecomposable injective modules $I$ to the last term $P$ in the minimal projective resolution of $I$, which is an indecomposable projective module. 
When $A$ is a finite dimensional representation-finite algebra and $M$ is the direct sum of all indecomposable $A$-modules, then the \emph{Auslander algebra} of $A$ is defined as the endomorphism ring $B=\End_A(M)$.
By the celebrated Auslander correspondence \cite{AusQu}, Auslander algebras have global dimension at most 2 and dominant dimension at least two and thus are Auslander--Gorenstein. Since the simple $B$-modules correspond bijectively to the indecomposable $A$-modules, the Auslander--Reiten bijection describes a bijection on the indecomposable $A$-modules. It was shown in \cite{MTY} in slightly more generality that the Auslander--Reiten bijection on $B$ corresponds to the Auslander--Reiten translate $\tau$ acting on the indecomposable non-projective $A$-modules and acting as the Nakayama functor $\nu$ on indecomposable projective modules.
Thus, the Auslander--Reiten bijection recovers the classical Auslander--Reiten translate when specialised to Auslander algebras.

Let $A$ be a ring and $M$ an $A$-module. An important role in representation theory, algebraic topology, homological algebra and several other fields is played by the Yoneda algebra $\bigoplus\limits_{i=0}^{\infty}{\Ext_A^i(M,M)}$ of $M$. For example when $A$ is a Koszul algebra with radical $J$, the quadratic dual is given by taking $M=A/J$ and that leads to the important theory of Koszul duality, see for example \cite{BGS}.

In this article we want to focus on the Auslander--Gorenstein property of \emph{Auslander--Yoneda} algebras $Y(A)$ for finite dimensional algebras $A$, that are defined by taking for $M$ the direct sum of all indecomposable $A$-modules up to isomorphism.
This ring was first studied from a homological viewpoint by Hanihara \cite{Ha}, where he described several important homological properties.

In this article we want to answer the following two questions:
\begin{question}
Let $A$ be a finite dimensional algebra of finite global dimension. 
When is the Auslander--Yoneda algebra of $A$ Auslander--Gorenstein?
And what is the Auslander--Reiten bijection in that case?
\end{question}

We will give complete answers to those questions which give a new link to the study of spherical modules that was initiated by Auslander and Bridger in \cite{AB}.
Following \cite{AB}, a module $M$ is called \emph{$n$-spherical} if $\pdim M=n$ and $\Ext_A^i(M,A)=0$ for all $1 \leq i <n$. We simply call $M$ spherical if it is $n$-spherical for some $n \geq 0$. $M$ is called \emph{$n$-perfect} if it is $n$-spherical and additionally $\Hom_A(M,A)=0$, we refer for example to \cite{BH} for the relevance of perfect modules in commutative algebra. Dually, we define a module $M$ of injective dimension $n$ to be \emph{$n$-cospherical} if $\Ext_A^i(DA,M)=0$ for all $1 \leq i <n$. $M$ is called \emph{$n$-coperfect} if it is $n$-cospherical and $\Hom_A(D(A),M)=0$.

The significance of these vanishing conditions is already visible in commutative algebra and algebraic geometry: the Auslander--Bridger spherical filtration, later appearing in dual form in work of Horrocks on vector bundles and in the syzygy-theoretic work of Evans--Griffith, organizes modules according to the degree in which their Ext-cohomology is concentrated \cite{AB,Ho1,Ho2,EG1,EG2}. Thus spherical modules provide a homological language for phenomena governed by grade, syzygies, reflexivity and vector-bundle extension problems. More recently, they have reappeared in connection with quasi-Auslander--Gorenstein conditions, delooping level and finitistic dimension, stable categories of torsion-free modules, and Auslander-type approximation theory \cite{G,Hu,O1,O2}.

We introduce the following notion.
\begin{definition}
We call a finite dimensional algebra $A$ \emph{(co)spherical} if every indecomposable (left) right $A$-module is spherical and we call $A$ \emph{bispherical} if every indecomposable left and right module is spherical. %Note that this implies that $A$ has finite global dimension.
\end{definition}
By definition, each (co)spherical algebra has finite global dimension. For example, every hereditary finite dimensional algebra is bispherical, but there are many other classes of examples as we will see later.
For the next theorem, we remark that a graded ring is Auslander--Gorenstein in the usual sense if and only if it is Auslander--Gorenstein in the graded sense, we refer for example to \cite[Section 3.3]{LVO} for details, where a proof is given under the assumption of finite global dimension, but a similar proof works also without this assumption. See also \cite[Proposition 1.3 and 1.4]{Gor}.
Using those definitions, we can answer the previous question as in the following theorem:

%For a finite dimensional algebra $A$ of finite global dimension, consider the Yoneda category
%\[ \Y:=\add\{M[i]\mid M\in\mod A, i\in\Z\}\subset\Db(\mod A). \]
\begin{theorem}[Theorem \ref{Yoneda-AG}]
Let $A$ be a finite dimensional algebra of finite global dimension which is representation-finite. Let $M$ be the basic additive generator for $\mod A$ and $Y(A)=\bigoplus_{i\geq0}\Ext^i_A(M,M)$ the Auslander--Yoneda algebra.
\begin{enumerate}
\item $Y(A)$ is (graded) Auslander--Gorenstein if and only if $A$ is bispherical.
\item In this case, the Auslander--Reiten permutation $F$ is given as follows: for each indecomposable $N\in\mod A$, we have
\[  FN=\begin{cases}
 \nu N & \text{ if } N \text{ is projective},\\
 \tau_nN(n) & \text{ if } N \text{ is $n$-perfect},\\
 \Omega^{-1} \tau N &\text{ if } N \text{ is non-projective and } \Hom_A(N,A)\neq0,
 \end{cases}
\]
where $(n)$ is the degree shift of graded $Y(A)$-modules.
\end{enumerate}
\end{theorem}
In fact, we prove the above result for arbitrary finite dimensional algebras of finite global dimension (not necessarily representation-finite) in terms of {\it Yoneda categories}, we refer to Theorem \ref{Yoneda-AG} for details.

It is a non-trivial task to construct examples of bispherical algebras.
Our next result, in fact, gives a large class of examples of spherical algebras. Recall that for each $m\geq1$, the {\it $m$-replicated algebra} of a finite dimensional algebra $A$ is defined by the $(m+1)\times(m+1)$-matrix algebra
\[ A^{(m)}=
\begin{pmatrix} A&0& 0&\cdots& 0& 0 \\
DA&A&0&\cdots&0&0 \\
0&DA&A&\cdots&0&0 \\
\vdots&\vdots&\vdots&\ddots&\vdots&\vdots \\
0&0&0&\cdots&A&0\\
0&0&0&\cdots&DA&A
\end{pmatrix}. \]
with $DA=\Hom_K(A,K)$, the dual of the regular module having trivial multiplication. The concept of $m$-replicated algebras was first defined and studied in \cite{AI}. Since then the study of replicated algebras found many connections and applications for example related to the study of cluster theory and fractionally Calabi-Yau algebras, we refer for example to \cite{ABST,CIM,CIM2,X}.
\begin{theorem}[Theorem \ref{replication}]
Let $A$ be a hereditary algebra. Then the $m$-replicated algebra of $A$ is also bispherical for any $m \geq 1$.
\end{theorem}
In particular, we get infinitely many examples of spherical algebras of arbitrarily high global dimension from a given hereditary algebra.
This naturally motivates the following problem:

\begin{problem}
Classify the spherical finite dimensional algebras.
\end{problem}
Of course, this is equivalent to classifying the cospherical finite dimensional algebras because $A$ is spherical if and only if $A^{op}$ is cospherical.

We give some partial progress to this problem in the following.
Our general result gives a restriction for a finite dimensional algebra to be (co)spherical in terms of the existence of a natural split torsion pair.
Define the subcategories by
\[
\begin{aligned}
\T&:=\{ M\in\mod A\mid \Hom_A(M,A)=0\}, \\
\F&:=\{ M\in\mod A\mid M \text{ is a submodule of a projective module}\}.
\end{aligned}
\]
Our characterization of cospherical algebras is as follows.
\begin{theorem}[{\cref{thm::characterisationcospherical}}]\label{intro:char}
Let $A$ be a finite dimensional algebra of finite global dimension.
Then the following are equivalent.
   \begin{enumerate}[label = (\roman*)]
        \item A is cospherical.
        \item The pair $(\T, \F)$ is a split torsion pair.
        %\item Every indecomposable module $M \in \mod A$ of injective dimension $\idim M \geq 2$ satsifies $\Ext^1_A(DA,M)=0$.
    \end{enumerate}
    %Moreover, in this case any module in in $\F$ is either projetcive or $(n-1)$-torsion-free for $n:=\idim\tau M\geq2$.
\end{theorem}
We refer to Theorem \ref{thm::characterisationcospherical} for a further equivalent condition and to Proposition \ref{prop::cosphericalcharacterisation} for more information on the torsion pair.

Now, the above splitting torsion pair gives a restriction on the shape of the Auslander--Reiten quiver. This leads to the following result.
Recall that a {\it cycle} in $\mod A$ is a sequence of indecomposable modules $(X_0,X_1,\cdots,X_n)$ such that $\rad(X_{i-1},X_{i})\neq0$ for $1\leq i\leq n$ and $X_0\simeq X_n$. An algebra $A$ is {\it directed} if $\mod A$ has no cycles.
  
\begin{theorem}[{\cref{theorem:directed}}]\label{theorem:directedIntro}
Each representation-finite (co)spherical algebra is directed.
In particular, its Gabriel quiver is acyclic.
%Let $A$ be a representation-finite (co)spherical algebra. Then the category $\mod A$ is directed.
\end{theorem}

This leads to a full classification for Nakayama algebras:
\begin{theorem}[{\cref{coro:LinearNakayama}}]
Let $A$ be a connected Nakayama algebra.
Then the following are equivalent:
\begin{enumerate}[label = (\roman*)]
    \item $A$ is cospherical.
    \item $A$ has a linear quiver and weakly decreasing Kupisch series.
\end{enumerate}
In particular, (co)spherical Nakayama algebras with $n \geq 2$ simple modules are enumerated by $2^{n-2}$.
\end{theorem}
\begin{theorem}[{\cref{bisphericalNakayamaalgebras}}]
Let $A$ be a connected Nakayama algebra. The following are equivalent.
\begin{enumerate}[label = (\roman*)]
\item The Auslander--Yoneda algebra of $A$ is Auslander--Gorenstein.
\item $A$ is bispherical.
\item $A\simeq kQ/\rad ^{\ell}kQ$ for some $\ell \geq 2$, where $Q$ is the linearly oriented quiver of type $A$.
\end{enumerate}
\end{theorem}

\cref{theorem:directedIntro} motivates the following question.

\begin{question}
Are there bispherical algebras whose Gabriel quiver contains a cycle?
\end{question}

{We now turn to the final application of our notion of spherical algebras.} 
In \cite[Question 3.17]{V}, motivated by the study of the Auslander regular property of rings in the setting of Iwasawa algebras
of $p$-adic Lie groups, Venjakob posed %the following question:
a question about the structure of modules over an Auslander regular ring $\La$ in a certain quotient category of $\mod\La$.
This can be stated as follows in the language of torsion pairs.
\begin{question}
Let $\La$ be an Auslander regular ring.
Is the torsion pair $(q(\E),\refl\La)$ always split?
\end{question}
We refer to \cref{section:7} for the precise definition of $(q(\E),\refl\La)$.
We note that the torsion pair $(q(\E),\refl\La)$ can be identified with the image of $(\T,\F)$ above in the quotient category (see \cref{lemma::T=E}).
Our characterization \cref{intro:char} then gives 
%We give 
the following positive answer in the cospherical case:
\begin{theorem}[{\cref{splitARcorollary}}]
Let $A$ be a finite dimensional cospherical Auslander regular algebra. Then $(q(\E),\refl A)$ is split.
\end{theorem}
This result, together with our classification of cospherical Nakayama algebras, motivated us to test the question by Venjakob for general Auslander regular algebras that are possibly non-cospherical. We found an explicit counterexample inside the class of Nakayama algebras, we refer to Proposition \ref{Venjakobcounterexample}.

\section{Basic definitions and properties}
We assume the reader is familiar with representation theory and homological algebra of finite dimensional algebras and refer for example to the textbooks \cite{ASS,Hap,Kr,SkoYam}. Unless otherwise stated, we work with finite dimensional $K$-algebras with $K$ a field and assume modules are finitely generated right modules.
\begin{definition} \cite{AB}
Let $A$ be a finite dimensional algebra and $M\in\mod A$.
\begin{enumerate}
\item We say that a module $M$ of finite projective dimension is {\it spherical} if $\Ext^i_A(M,A)=0$ unless $i\in\{0, \pdim M\}$.
\item We say that a module $M$ of finite injective dimension is {\it cospherical} if $\Ext^i_A(DA,M)=0$ unless $i\in\{0,\idim M\}$.
\end{enumerate}
\end{definition}

\begin{definition}
We call a finite dimensional algebra {\it spherical} (resp. {\it cospherical}) if every finitely generated indecomposable module is spherical (resp. cospherical). A {\it bispherical algebra} is a finite dimensional algebra which is both spherical and cospherical.
\end{definition}
In particular, each (co)spherical algebra has finite global dimension.

\begin{definition}
    Let $M$ be an $A$-module. 
    The \emph{grade} of $M$ is defined as
    \begin{equation*}
        \operatorname{grade} M := \operatorname{grade}_{A} M := \min \{ i \geq 0 \mid \Ext^i_A(M,A) \neq 0\}.
    \end{equation*}
    It is defined to be $\infty$ if no such $i$ exists. 
    The \emph{reduced grade} is the minimal $i \geq 1$ such that $\Ext^i_A(M,A) \neq 0$.

    The \emph{strong grade} of $M$ is the infimum over all grades of submodules of $M$.
    \begin{equation*}
        \sgrade M :=  \sgrade_{A} M := \inf\{\, \operatorname{grade} N \mid N\subseteq M\,\}.
    \end{equation*}
\end{definition}
We refer for example to \cite{HI} for more on the grade, reduced grade and strong grade of modules.

Note that a finite dimensional algebra of finite global dimension is spherical if and only if for every indecomposable module its reduced grade is maximal.

\begin{lemma}\label{lemma::sphericalcontrolsgradeofExt}
    Let $M$ be a spherical $A$-module. Then $\operatorname{grade} \Ext_A^i(M,A) \geq i$ for all $i \geq 0$.
    \begin{proof}
    Let $n:=\pdim M$. Since $M$ is spherical, we just have to prove that the left $A$-module $E:=\Ext_A^n(M,A)$ satisfies $\operatorname{grade} E\geq n$.
    Let 
    $$0 \rightarrow P_n \rightarrow \cdots \rightarrow P_0 \rightarrow M \rightarrow 0$$ be a minimal projective resolution of $M$ and apply the functor $\Hom_A(-,A)$ to it to get the exact sequence of left $A$-modules (here we use that $M$ is spherical to get exactness):
    $$0 \rightarrow M^{*} \rightarrow P_0^{*} \rightarrow \cdots \rightarrow P_n^{*} \rightarrow E \rightarrow 0.$$
    We now compute $\Ext_{A^{op}}^j(E,A)$ by applying $\Hom_{A^{op}}(-,A)$ to the complex above to get the complex
    $$0 \rightarrow P_n^{**} \rightarrow \cdots \rightarrow P_0^{**} \rightarrow M^{**}\rightarrow0$$
    which computes $\Ext^j_{A^{op}}(E,A)$.
    Now we use that the $P_i$ are projective and thus reflexive: $P_i \cong P_i^{**}$ and the complex is isomorphic to 
    $$0 \rightarrow P_n \rightarrow \cdots \rightarrow P_0 \rightarrow M^{**}\rightarrow0,$$ the original complex up to $P_0^{**}$, which is exact.
    Thus $\Ext_{A^{op}}^j(E,A)=0$ for $0 \leq j<n$ and thus $\operatorname{grade} E \geq n$.
    \end{proof}
\end{lemma}

In Theorem \ref{Yoneda-AG} we show that a representation-finite algebra is bispherical if and only if its Auslander--Yoneda algebra is Auslander--Gorenstein. 
On the other hand, Lemma \ref{lemma::sphericalcontrolsgradeofExt} implies that a spherical algebra is already right quasi-Auslander--Gorenstein, see Proposition \ref{qAG} below.

\begin{definition}
An algebra $A$ is called \emph{right quasi-Auslander--Gorenstein} if $\idim P_i \leq i+1$ for all $i \geq 0$ when 
$$\cdots \rightarrow P_{n+1} \rightarrow P_n \rightarrow \cdots \rightarrow P_0 \rightarrow D(A) \rightarrow 0$$
is a minimal projective resolution of $D(A)$.
Dually, one defines left quasi-Auslander--Gorenstein algebras, which is also equivalent to $A^{op}$ being right quasi-Auslander--Gorenstein. An algebra is called \emph{quasi-Auslander--Gorenstein} if it is left and right quasi-Auslander--Gorenstein. 
\end{definition}

Notice that in contrast to being Auslander--Gorenstein, being right quasi-Auslander--Gorenstein is not left-right symmetric.

Right quasi-Auslander--Gorenstein algebras have been used by Auslander and Reiten in \cite{AR2} to characterize algebras with extension-closed syzygy categories and they have several equivalent characterizations and good properties, see also for example \cite{Hu}. Another important application is that the spherical filtration theorem holds for all modules for right quasi-Auslander--Gorenstein algebras, see for example \cite{AB} and \cite[Theorem 2.7]{G}.

\begin{proposition}\label{qAG}
Let $A$ be a spherical algebra. Then $A$ is right quasi-Auslander--Gorenstein.
\end{proposition}
\begin{proof}
By \cite[Theorem 0.1]{AR2} $A$ is right quasi-Auslander--Gorenstein if and only if $\operatorname{grade} \Ext_A^i(M,A) \geq i$ as $A^\op$-modules for all $i \geq 0$ and all $A$-modules $M$. 
By additivity of $\Ext^i_A(-,A)$, it is enough to show $\operatorname{grade} \Ext_A^i(M,A) \geq i$ for all $i \geq 0$ for indecomposable modules $M$.
Since by assumption $A$ is spherical, we get that any indecomposable $A$-module $M$ is spherical and thus satisfies the condition by Lemma \ref{lemma::sphericalcontrolsgradeofExt}.
\end{proof}

\begin{corollary}
Let $A$ be a bispherical algebra. Then $A$ is quasi-Auslander--Gorenstein.
\end{corollary}

From this, our examples of (bi)spherical algebras give two non-trivial classes of quasi-Auslander--Gorenstein-algebras, see Corollary \ref{cor::somelinearNakayamaalgebrasareqAG} and Corollary \ref{cor::replciatedalgebraqAG}. 

\section{The split torsion pair for (co)spherical algebras}
We consider the following two subcategories:
\[
\begin{aligned}
\T&:=\{ M\in\mod A\mid \Hom_A(M,A)=0\}, \\
\F&:=\{ M\in\mod A\mid M \text{ is a submodule of a projective module}\}.
\end{aligned}
\]
The main result of this section is that $(\T, \F)$ being a split torsion pair is a characterization of cosphericalness. %as the following Theorem shows.

\begin{theorem}\label{thm::characterisationcospherical}
    Let $A$ be a finite dimensional algebra of finite global dimension.
    Then the following are equivalent.
    \begin{enumerate}[label = (\roman*)]
        \item A is cospherical.
        \item The pair $(\T, \F)$ is a split torsion pair.
        \item Every indecomposable module $M \in \mod A$ of injective dimension $\idim M \geq 2$ satisfies $\Ext^1_A(DA,M)=0$.
    \end{enumerate}
    Moreover, in this case any indecomposable module $M \in \F$ is either projective or $(n-1)$-torsion-free for $n:=\idim\tau M\geq2$.
\end{theorem}

\begin{remark}
    Notice that in general $(\T, \F)$ is not a torsion pair if it is not split which can be seen in Example \ref{example::nondiectedalgerbawithoutcyclictauorbit} below.
    In general, $ \T = \ ^\perp \F$ so that $(\T, \F)$ is a torsion pair if and only if $\F$ is a torsion free class. 
    This is the case if and only if $\F$ is extension-closed and by \cite[Theorem 0.1]{AR2} this is equivalent to $A$ being right $1$-quasi-Auslander--Gorenstein.
\end{remark}

Recall that for $\La$ a (two-sided) noetherian ring and $M\in\mod\La$, $M$ is {\it $k$-torsion-free} if $\Ext^i_{\La^\op}(\Tr M,\La)=0$ for $1\leq i\leq k$. This is equivalent to saying that there is an exact sequence
\[ \xymatrix{ 0\ar[r]& M\ar[r]& P^0\ar[r]&P^1\ar[r]&\cdots\ar[r]& P^{k-1} } \]
with $P^i\in\proj\La$ such that the sequence $0\to \nu M\to \nu P^0\to \nu P^1\to\cdots\to \nu P^{k-1}$ is exact. We refer for example to \cite{AB} and \cite{AR1,AR2} for more on $k$-torsion-free modules. 

We first show that cosphericalness makes $(\T,\F)$ a split torsion pair. 
\begin{proposition}\label{prop::cosphericalcharacterisation}
Let $A$ be a finite dimensional algebra of finite global dimension. 
If $A$ is cospherical, then each indecomposable $A$-module $M$ satisfies one of the following.
\begin{itemize}
\item[(t)] $\Hom_A(M,A)=0$.
\item[(f)] $M$ is projective, or $n:=\idim\tau M\geq 2$ and $M$ is $(n-1)$-torsion-free.
\end{itemize}
Consequently, the pair $(\T,\F)$ forms a split torsion pair.
\end{proposition}
\begin{proof}
    Let $M\in\mod{A}$ be non-projective. Consider a minimal projective presentation 
    \[P_1\to P_0\to M\to 0\] 
    of $M$. Applying the Nakayama functor $\nu=D\Hom(-,A)$ we obtain an exact sequence
    \[0\to \tau M\to \nu(P_1)\to \nu(P_0)\to\nu(M)\to 0.\]
%Let $M\in\mod A$ be an indecomposable non-projective module. Let $P_1\to P_0\to M\to0$ be the minimal projective presentation, which yields the minimal injective presentation $0\to \tau M\to \nu P_1\to \nu P_0$ whose cokernel is $\nu M$.
	Putting $n:=\idim\tau M\geq1$, we may write its minimal injective resolution as %$0\to \nu M\to I^1\to I^2\to\cdots\to I^{m-1}\to 0$
	\[ \xymatrix@R=1mm@C=3mm{
		&\nu P_1\ar[rr]&&\nu P_0\ar[rr]\ar[dr]&& I^1\ar[rr]&&I^2\ar[rr]&&\cdots\ar[rr]&& I^{n-1}\\
		\tau M\ar[ur]&&&&\nu M\ar[ur]&&&&&&&. } \]
	Applying $\Hom_A(DA,-)$ we get a complex
	\begin{equation}\label{ext}
	\xymatrix{	P_1\ar[r]&P_0\ar[r]& \nu^{-1}I^1\ar[r]&\nu^{-1}I^2\ar[r]&\cdots\ar[r]& \nu^{-1}I^{n-1}	}
	 %\xymatrix@R=1mm@C=3mm{
	%	P_1\ar[rr]&&P_0\ar[rrr]\ar[dr]&&& \nu^{-1}I^1\ar[rr]&&\nu^{-1}I^2\ar[rr]&&\cdots\ar[rr]&& \nu^{-1}I^{n-1}\\
	%	&& & M\ar[r]&\nu^{-1}\nu M\ar[ur]& }
	\end{equation}
	which computes $\Ext^i_A(DA,\tau M)$.
	
	Suppose that $A$ is cospherical. Then the complex \eqref{ext} is acyclic except at the end terms.
	If $n=1$, then this means $\nu M=0$, so $M$ satisfies (t). If $n\geq2$, then the complex \eqref{ext} is acyclic at $P_0, \nu^{-1}I^1,\ldots,\nu^{-1}I^{n-2}$, which means $\Ext^i_A(DA,\tau M)=0$ for $1\leq i\leq n-1$. Dualizing, we see that $\Ext^i_{A^\op}(\Tr M,A)=0$ for $1\leq i\leq n-1$, thus $M$ is $(n-1)$-torsion-free.
	%Suppose conversely that each indecomposable $A$-module satisfies either (t) or (f). We show that $A$ is cospherical. Clearly, any injective module is cospherical, so it is enough to show that $\tau M$ is cospherical for each indecomposable $M\in\mod A$.
	%If $M$ satisfies (t), then $\nu M=0$, so $\idim \tau M\leq1$ and thus $\tau M$ is cospherical. Suppose that $M$ satisfies (f). Then $\Ext^i_A(DA,\tau M)=\Ext^i_{A^\op}(\Tr M,A)=0$ for $1\leq i\leq \idim\tau M-1$, and hence $\tau M$ is cospherical.

    Since each module satisfying (f) is in $\F$, we obtain that each indecomposable module lies either in $\T$ or $\F$. Also, since we have $\Hom_A(T,F)=0$ for each $T\in\T$ and $F\in\F$, we deduce that $(\T,\F)$ is a split torsion pair.
\end{proof}

\begin{remark}
Condition (f) means that $M$ is ``as torsion-free as possible''. Indeed, over an algebra of finite global dimension, if $M$ is $k$-torsion free, then one must have $k<\idim\tau M$.
\end{remark}

%Since every $1$-torsion free module is a submodule of a projective, we get the following Corollary.
%
%\begin{corollary}\label{coro:SplitTorsion}
%Let $A$ be a cospherical algebra. Then the pair $(\T,\F)$ forms a split torsion pair.
%\end{corollary}

For the proof of Theorem \ref{thm::characterisationcospherical} we need to recall the dual of \cite[Proposition 2.43]{AB}.

\begin{proposition}[{\cite[Proposition 2.43]{AB}}] \label{prop::cosyzygy}
    Let $A$ be a finite dimensional algebra and $N$ an $A$-module.
    If $\Ext_A^i(D(A),N)=0$ for $1 \leq i \leq k$, then the functor $\Omega^{-i} \colon \overline{\mod} A \to \overline{\mod} A$ induces an isomorphism 
    \begin{equation*}
        \overline{\Hom}_A(M,N) \cong \overline{\Hom}_A(\Omega^{-i}(M), \Omega^{-i}(N))
    \end{equation*} 
    for $1 \leq i \leq k$ and $M \in \mod A$.
\end{proposition}

%For the convenience of the reader, we include a short a proof of our case.
%\begin{proof}
%Let $M \hookrightarrow I_M$ be the injective envelope of $M$ and fix the short exact sequence.
%\begin{equation*}
%    0 \to M \to I_M \to  \Omega^{-1} M \to 0
%\end{equation*}
%Since $I_M$ is the injective envelope, we get that
%\begin{equation*}
%    \overline{\Hom}_A(M,M) \cong \coker(\Hom_A(I_M, M) \to \Hom_A(M,M)) \cong \Ext^1(\Omega^{-1} M, M) 
%\end{equation*}
%where the second isomorphism is induced from the long exact sequence obtain from the functor $\Hom(-, M)$.
%On the other hand we have a canonical isomorphism 
%\begin{equation*}
%    \Ext^1(\Omega^{-1}M, M) \cong \coker (\Hom_A(\Omega^{-1} M, I_M) \to \Hom_A(\Omega^{-1}M, \Omega^-M))
%\end{equation*}
%from the long exact sequence obtained from $\Hom(\Omega^{-1}. -)$.
%To see that the later cokernel is $\overline{\Hom}_A(\Omega^{-1}(M), \Omega^{-1}(M))$ it remains to show to any morphism $I \to \Omega^{-1} M$ where $I$ is injective factors through $I_M$ which we obtain from the long exact sequence
%\begin{equation*}
%    0 \to \Hom_A(I,M) \to \Hom_A(I, I_M) \to \Hom_A(I, \Omega^{-1} M) \to \Ext^1(I, M) = 0
%\end{equation*}
%of the fucntor $\Hom_A(I, -)$.
%Since the isomorphism $\overline{\Hom}_A(M,M) \cong \overline{\Hom}_A(\Omega^{-1}(M), \Omega^{-1}(M))$ constructed above is the one induced from the functor $\Omega^{-1} \colon \overline{\mod} A \to \overline{\mod} A$, it is also an isomorphism of of algebras.
%The case $k>1$ now follows from induction.
%\end{proof}

\begin{corollary}\label{coro::Ausladnerlemma}
    Let $A$ be a finite dimensional algebra and $M$ a non-injective indecomposable $A$-module with $\Ext_A^1(D(A),M)=0$.
    Then $\Omega^{-1} M \cong N \oplus I$ where $N$ is indecomposable and $I$ is injective. 
    \begin{proof}
        Since $\mod A$ is a Krull--Schmidt category, we have that $M$ is indecomposable if and only if $\End_A(M)$ is a local ring and moreover, $\overline{\mod} A$ is also a Krull-Schmidt category. 
        Hence, if $M$ is indecomposable, then $\overline{\End}(M) \cong \overline{\End}(\Omega^{-1}M)$ is local or $0$.
        But then $\Omega^{-1} M$ can have at most one indecomposable non-injective direct summand. 
    \end{proof}
\end{corollary}

\begin{proof}[Proof of Theorem \ref{thm::characterisationcospherical}]
    We have seen already in Proposition \ref{prop::cosphericalcharacterisation} that if $A$ is cospherical, then $(\T, \F)$ is a split torsion pair. 

    Assume now that $(\T, \F)$ is a split torsion pair.
    As in the proof of Proposition \ref{prop::cosphericalcharacterisation}, notice that a minimal projective presentation $P_1\to P_0\to M\to0$ of a non-projective module $M \in \mod A$ yields the following exact sequence.
    \[0\to \tau M\to \nu(P_1)\to \nu(P_0)\to\nu(M)\to 0.\]
    It follows that $M\in\T$ if and only if $\nu(M)=0$ if and only if $\idim\tau M\leq 1$, i.e.
    \[\T=\{M\in\mod A\mid \Hom_A(M,A)=0\}=\{M\in\mod A\backslash \proj{A}\mid \idim \tau M\leq 1\}.\]
    Consequently, if $X$ is indecomposable with $\idim X \geq 2$, then $X=\tau M$ for some non-projective $M\in \F$.
    Now, $M$ is a submodule of a projective if and only if $ev:M\to (M^*)^*$ is a monomorphism. By \cite[Proposition 2.6]{AB}, this is the case if and only if $\ker(ev)=\Ext^1_{A^{op}}(\Tr M,A)=\Ext^1_A(DA,\tau M)$ is $0$.

    Finally, assume $\Ext^1_A(DA,M) = 0$ if $\idim M \geq 2$.
    We prove by induction on $n=\idim M$ that $\Ext^i_A(DA,M)=0$ for $1\leq i\leq n-1$. 
    %There is nothing to prove for $n=1$, so assume $n\geq2$.
    By assumption, any module with $\idim M \in \{0,1, 2\}$ is automatically cospherical.
    So assume now that $\idim M = n+1 > 2$. 
    By Corollary \ref{coro::Ausladnerlemma} we know that $\Omega^{-1} M \cong N \oplus I$ for $N$ indecomposable and $I$ injective. 
    Hence, $\idim N = n$ and thus $N$ is cospherical by the induction hypothesis.
    Since $I$ is injective, we get that also $\Omega^{-1}M$ is cospherical.
    In particular, we get that
    \begin{equation*}
        \Ext^i_A(DA, M) \cong \Ext^{i-1}_A(DA, \Omega^{-1} M) = 0
    \end{equation*}
    for $2 \leq i \leq n$, proving that $M$ is also cospherical.
    
    The final assertion is a direct application of Proposition \ref{prop::cosphericalcharacterisation}.
\end{proof}

\section{Auslander--Gorenstein property of Auslander--Yoneda algebras}

We first extend the notion of being Auslander--Gorenstein from finite dimensional algebras to dualizing varieties.

Recall that a {\it dualizing variety} is a $K$-linear category $\C$ such that $D=\Hom_K(-,K)$ induces a duality $\mod\C\leftrightarrow\mod\C^\op$.
It follows from the definition that for a dualizing variety $\C$, the category $\mod\C$ is $K$-linear, $\Hom$-finite, Krull--Schmidt with enough projectives $\{\C(-,C)\mid C\in\C\}$ and enough injectives $\{D\C(C,-)\mid C\in\C\}$. We refer for example to \cite{AR3} for more on dualizing varieties.
\begin{definition}
Let $\C$ be a dualizing variety. For each $C\in\C$, consider the minimal injective resolution
\[ \xymatrix@!R=1mm{
	0\ar[r]&\C(-,C)\ar[r]&I^0\ar[r]&I^1\ar[r]&\cdots \\
	0\ar[r]&\C(C,-)\ar[r]&J^0\ar[r]&J^1\ar[r]&\cdots} \]
of projective $\C$- and $\C^\op$-modules. We say that $\C$ is {\it $n$-Gorenstein} if $\pdim I^i\leq i$ in $\mod\C$ and $\pdim J^i\leq i$ in $\mod\C^\op$ for all $0\leq i<n$. We say that $\C$ is {\it Auslander--Gorenstein} if $\C(-,C)$ and $\C(C,-)$ have finite injective dimension for each $C\in\C$ and $\C$ is $n$-Gorenstein for all $n>0$.
\end{definition}
Let $\C$ be an idempotent-complete dualizing variety which is Auslander--Gorenstein. For each indecomposable $C\in\C$, let
\[ \xymatrix{ 0\ar[r]&\C(-,C)\ar[r]&I^0\ar[r]&I^1\ar[r]&\cdots\ar[r]& I^n\ar[r]& 0} \]
be the minimal injective resolution. In this case, it follows as in \cite[Chapter 5]{AR1} that the last term $I^n$ is indecomposable, and hence is isomorphic to $D\C(C^\prime,-)$ for some indecomposable $C^\prime\in\C$. The correspondence $C\mapsto C^\prime$ yields a bijection on the set of isomorphism classes of indecomposables, which is called the {\it Auslander--Reiten bijection}.

\bigskip

Let $A$ be a finite dimensional algebra, and consider the full subcategory
\[ \Y:=\add\{M[i]\mid M\in\mod A, i\in\Z\}\subset\Db(\mod A) \]
of the derived category, called the {\it Yoneda category}, which was first considered in \cite[Chapitre~III.3.3.2]{Ver}. 
Recall from \cite[Proposition~4.2]{Ha1} that if $A$ is representation-finite, then 
\begin{align*}
    \Y \cong \proj^\Z Y(A) && \text{and} && \mod\Y \cong \mod^\Z Y(A)
\end{align*}
so that the Yoneda category generalizes the Auslander--Yoneda algebra $Y(A) := \Ext^*_A(M,M)$, where $M$ is a (basic) additive generator of $\mod A$, in terms of its (graded) representation theory.

By \cite[Theorem~3.1]{Ha}, we know that $\Y$ is a dualizing variety which is $1$-Iwanaga--Gorenstein when $A$ has finite global dimension. The main result of this section is the following characterization of $\Y$ being Auslander--Gorenstein, and the description of the Auslander--Reiten bijection.
For $L\in\mod A$, we denote by $L_{\ij}$ the maximal injective direct summand of $L$, and by $L_{\ninj}$ the complement, so we have $L=L_{\ij}\oplus L_{\ninj}$.
\begin{theorem}\label{Yoneda-AG}
Let $A$ be a finite dimensional algebra of finite global dimension.
\begin{enumerate}
\item\label{iff} $\Y$ is Auslander--Gorenstein if and only if $A$ is bispherical.
\suspend{enumerate}
Suppose in what follows that the conditions in (\ref{iff}) are satisfied.
\resume{enumerate}
\item\label{resol} For each indecomposable $n$-spherical module $M\in\mod A$, the minimal projective resolution of the injective module $D\Y(M,-)$ is given by $\Y(-,\nu M)\xsimeq D\Y(M,-)$ 
%\[ \xymatrix{ 	&0\ar[r]&\Y(-,\nu M)\ar[r]&D\Y(M,-)\ar[r]& 0 } \]
if $M$ is projective, and by
\[ \xymatrix@C=6mm{ &0\ar[r]& \Y(-,(\Om^{-1}\tau M)_\ninj)\ar[r]&\Y(-,(\nu P_0)/(\Om^{-1}\tau M)_\ij)\oplus \Y(-,\tau_nM[n])\ar[r]& D\Y(M,-)\ar[r]& 0} \]
if $M$ is non-projective.
\item\label{ARbij} %Suppose that $\Y$ is Auslander--Gorenstein. 
The Auslander--Reiten permutation $F$ is given as follows: for $M\in\ind \Y$ of projective dimension $n$, we have
	\[  FM=\begin{cases}
		\nu M & \text{ if } M \text{ is a shift of a projective module},\\
		\tau_nM[n] & \text{ if } M \text{ is a shift of an $n$-perfect module},\\
		\Omega^{-1}\tau M=\Omega\nu M &\text{ if } M \text{ is a shift of a non-projective module such that } \Hom_A(M,A)\neq0.
	\end{cases}
	\]
\end{enumerate}
\end{theorem}

We start our discussion with a general observation on complexes. 
Recall that $X\in\rD(\Mod A)$ is a {\it stalk complex} if it is isomorphic to some shifts of modules: $X\simeq \bigoplus_{i\in\Z}H^iX[-i]$ in $\rD(\Mod A)$.
For a complex $X\in\rD(\Mod A)$, we denote by 
\[ ZX:=\bigoplus_{i\in\Z}Z^iX[-i], \text{ and } BX:=\bigoplus_{i\in\Z}B^iX[-i] \]
the stalk complexes of cocycles and coboundaries, respectively. We have an exact sequence
\[ \xymatrix{ 0\ar[r]& ZX\ar[r]& X\ar[r]& BX[1]\ar[r]& 0} \]
of complexes, which yields a triangle
\[ \xymatrix{ BX\ar[r]^\del&ZX\ar[r]& X } \]
in the derived category $\rD(\Mod A)$.
Let us first describe the connecting morphism $\del\colon BX\to ZX$ above. For each $i\in\Z$, we denote by $\iota$ the inclusion from $B^iX$ to $Z^iX$ or to $X^i$. Also, we let $\del_i\in\Hom_{\rD(\Mod A)}(B^iX,Z^{i-1}X[1])\simeq\Ext^1_A(B^iX,Z^{i-1}X)$ be the morphism corresponding to the extension $0\to Z^{i-1}X\to X^i\to B^iX\to 0$.
\begin{lemma}\label{BZ}
The restriction of $\del$ to the summand $B^iX[-i]$ is induced by $\begin{pmatrix} \iota\\ -\del_i\end{pmatrix}\colon B^iX\to Z^iX\oplus Z^{i-1}X[1]$.
\end{lemma}
\begin{proof}
	We replace the inclusion $ZX\to X$ by a surjection. Note that for each $i\in\Z$, the cocycle $Z^iX[-i]$ is quasi-isomorphic to the complex $\widetilde{Z_i}:=(X^i\to B^{i+1}X)$ with terms in degree $i$ and $i+1$. Then the inclusion $ZX\to X$ becomes a map $\bigoplus_{i\in\Z}\widetilde{Z_i}\to X$, which is surjective with kernel $BX$ as described in the diagram below.
	\[ \xymatrix{
		\cdots\ar[r]& B^{i-1}X\ar[r]\ar[d]& B^{i}X\ar[r]\ar[d]&B^{i+1}X\ar[r]\ar[d]&\cdots\\
		\cdots\ar[r]& B^{i-1}X \oplus X^{i-1}\ar[r]\ar[d]_-{\begin{pmatrix} \iota&1\end{pmatrix}}& B^{i}X\oplus X^{i}\ar[r]\ar[d]_-{\begin{pmatrix} \iota&1\end{pmatrix}}& B^{i+1}X\oplus X^{i+1}\ar[r]\ar[d]_-{\begin{pmatrix} \iota&1\end{pmatrix}}&\cdots\\
		\cdots\ar[r]& X^{i-1}\ar[r]& X^{i}\ar[r]&X^{i+1}\ar[r]&\cdots } \]
	It follows that at each degree $i$, the map $BX\to\bigoplus_{i\in\Z}\widetilde{Z_i}$ is given by $\begin{pmatrix} -1\\\iota\end{pmatrix}\colon B^iX\to B^iX\oplus X^i$, which corresponds to $B^iX\to\widetilde{Z_{i-1}}\oplus \widetilde{Z_i}$, as desired.
\end{proof}

Now we return to our setting of Yoneda categories.
\begin{lemma}\label{pr-in}
Let $X\in\Db(\mod A)$.
\begin{enumerate}
\item\label{pr} $\Y(-,X)$ is projective if and only if $X$ is a stalk complex.
\item\label{in} $\Y(-,X)$ is injective if and only if $\RHom_A(DA,X)$ is a stalk complex.
\end{enumerate}
\end{lemma}
\begin{proof}
	(1)  This is \cite[Lemma 3.5]{Ha}.
	
	(2)  We have $D\Y(-,X)=\Y(\RHom_A(DA,X),-)$ by Serre duality. Then the claim follows by applying (1) to $A^\op$.
\end{proof}

Following \cite[Section 3.1]{Ha}, we construct a projective resolution of each injective $\Y$-module.
Let $M\in\mod A$ be an indecomposable module, say, of projective dimension $n$. The injective $\Y$-module corresponding to $M$ is $D\Y(M,-)=\Y(-,M\lotimes_ADA)$ by Serre duality. If $0\to P_n\to\cdots P_1\to P_0\to M\to 0$ is the minimal projective resolution, the complex $M\lotimes_ADA$ is presented by the complex
\[ X:=\left( \xymatrix{ \nu P_n\ar[r]&\cdots\ar[r]&\nu P_1\ar[r]& \nu P_0 }\right)  \]
of injective modules. By \cite[Proposition 3.1]{Ha}, the sequence $BX\to ZX\to X\to BX[1]$ in $\Db(\mod A)$ induces a (potentially non-minimal) projective resolution  %$0\to\Y(-,BX)\to\Y(-,ZX)\to \Y(-,X)\to 0$ of $\Y(-,X)=D\Y(M,-)$.
%Note that we have $Z^0X=\nu P_0$, $Z^{-i}X=\tau_iM$, the $i$-Auslander--Reiten translation for $1\leq i\leq n$. Putting $\tau_0M:=\nu P_0$, the above resolution becomes
\begin{equation}\label{resolution}
\xymatrix{ 0\ar[r]& \disoplus_{i=0}^{n-1}\Y(-,B^{-i}X[i])\ar[r]^-\del&\disoplus_{i=0}^n\Y(-,Z^{-i}X[i])\ar[r]&D\Y(M,-)\ar[r]& 0 }.
\end{equation}
By Lemma \ref{BZ}, the first map $\del$ has components
\[ \xymatrix{ B^{-i}X\ar[r]& Z^{-i}X\oplus Z^{-i-1}X[1] } \]
%\[ \xymatrix@R=1mm{
%	&Z^{-i}X\\
%	B^{-i}X\ar[ur]\ar[dr]& \oplus\\
%	&Z^{-i-1}X[1] }  \]
for each $0\leq i\leq n-1$. Note also that $Z^0X=\nu P_0$ and $Z^{-i}X=\tau_iM$, the $i$-Auslander--Reiten translation.

Although we do not know the minimal projective resolution of the injective module, we have the following information. Recall that we write $L=L_\ij\oplus L_\ninj$ the decomposition of $L\in\mod A$ into injectives and non-injectives.
\begin{lemma}\label{B^0}
Let $M\in\mod A$ be an indecomposable module.
\begin{enumerate}
\item\label{ninj} The syzygy of $D\Y(M,-)$ contains $\Y(-,(B^0X)_\ninj)$ as a direct summand.
\item\label{ij} The following are equivalent.
\begin{enumerate}
\renewcommand{\labelenumii}{(\roman{enumii})}
\item $(B^0X)_\ninj=0$.
\item $M$ is projective or $\Hom_A(M,A)=0$.
\end{enumerate}
\end{enumerate}
\end{lemma}
\begin{proof}
	(1)  Along the decomposition $B^0X=(B^0X)_\ij\oplus (B^0X)_\ninj$, the map $B^0X\to \nu P_0\oplus \tau M[1]$ decomposes into a direct sum of the identity of $(B^0X)_\ij$ and $(B^0X)_\ninj\to (\nu P_0/(B^0X)_\ij)\oplus \tau M[1]$. Note that this is a radical map since the first factor is a map between non-injective modules and injective modules, and the second factor has degree $1$. It follows that the map $\delta$ in \eqref{resolution} restricted to $\Y(-,(B^0X)_\ninj)$ is a radical map, so the assertion follows.
	
	(2)  Suppose that $(B^0X)_\ninj=0$, in other words, the module $B^0X$ is injective. Then the map $f\colon \nu P_1\to \nu P_0$ is a direct sum of $\nu P_1\to B^0X$ and $0\to \nu P_0/B^0X$. By the minimality of the presentation $P_1\to P_0$ of $M$, we must have either $M$ is projective or $f$ is surjective. Since $\Coker f=\nu M=D\Hom_A(M,A)$, the latter case yields $\Hom_A(M,A)=0$.
	
	Suppose conversely that $M$ is projective or $\Hom_A(M,A)=0$. In the first case, we have $B^0X=0$. If $\Hom_A(M,A)=0$, then the map $\nu P_1\to \nu P_0$ is surjective, so $B^0X=\nu P_0$ is injective.
\end{proof}

We will regard $\tau_0M=\nu P_0$ and $B^{-n}X=0$ for $n=\pdim M$.
\begin{lemma}\label{tau_m}
Let $M\in\mod A$ be an indecomposable non-projective module, and put $m:=\min\{ i>0\mid \Ext^i_A(M,A)\neq0\}$, the reduced grade of $M$.
\begin{enumerate}
\item\label{proj} The projective cover of $D\Y(M,-)$ contains $\Y(-,\tau_mM[m])$ as a direct summand.
\item\label{syz} The syzygy of $D\Y(M,-)$ contains $\Y(-,B^{-m}X[m])$ as a direct summand.
\end{enumerate}
\end{lemma}
\begin{proof}
	If a part of the summand $\Y(-,\tau_mM[m])$ may be removed from the resolution \eqref{resolution}, a component of $\delta$ has to contain an identity summand, and it is necessarily $B^{-m}X[m]\to \tau_mM[m]$ since anything else is a radical map.
	Similarly, if a part of the summand $\Y(-,B^{-m}X[m])$ may be removed from the resolution \eqref{resolution}, the map $B^{-m}X\to \tau_mM$ has to contain an identity summand.
	But this is impossible since $\tau_mM$ is indecomposable.
\end{proof}

We are now ready to prove the main result of this section.
\begin{proof}[Proof of Theorem \ref{Yoneda-AG}]
(\ref{iff})  We prove that $\Y$ is Auslander--Gorenstein if $A$ is bispherical. Since each indecomposable module $M$ is spherical, the inclusion $B^{-i}X\to Z^{-i}X$ in the resolution above is the identity for $0<i\leq n-1$. It follows that the resolution \eqref{resolution} can be reduced to
\begin{equation}\label{close}
\xymatrix{ 0\ar[r]& \Y(-,B^0X)\ar[r]&\Y(-,\nu P_0)\oplus \Y(-,\tau_nM[n])\ar[r]& D\Y(M,-)\ar[r]& 0}.
\end{equation}
We now show that the middle term $\Y(-,\nu P_0)\oplus \Y(-,\tau_nM[n])$ is (projective-)injective.
Certainly, the module $\Y(-,\nu P_0)=D\Y(P_0,-)$ is injective. We prove that $\Y(-,\tau_nM)$ is injective. By \cref{pr-in}, we have to verify that $\RHom_A(DA,\tau_nM)$ is a stalk complex. Since $\tau_nM$ is indecomposable and cospherical, it suffices to show that $\Hom_A(DA,\tau_nM)=0$. Note that the first terms of the complex $X$ give an injective presentation of $\tau_nM$: $0\to \tau_nM\to \nu P_n\to \nu P_{n-1}$. Applying $\Hom_A(DA,-)$ we see that $\Hom_A(DA,\tau_nM)=0$ since $P_n\to P_{n-1}$ is injective. This shows that $\Y(-,\tau_nM)$ is injective, as desired.

We turn to the converse. We assume that $\Y$ is Auslander--Gorenstein and prove that each indecomposable $A$-module $M$ is spherical. Consider the projective resolution of the injective indecomposable $\Y$-module $D\Y(M,-)$ corresponding to $M$ as in \eqref{resolution}.
If $M$ is projective, then there is nothing to prove. Suppose that $M$ is non-projective and put
\[ m=\min\{i>0\mid \Ext^i_A(M,A)\neq0\}. \]
Since $\Y$ is Auslander--Gorenstein, the kernel of the projective cover of $D\Y(M,-)$ must be indecomposable, so Lemma \ref{B^0}(\ref{ninj}) and Lemma \ref{tau_m}(\ref{syz}) show that we have either $(B^0X)_\ninj=0$ or $B^{-m}X=0$.

In the second case, this means $m=\pdim M$, thus $M$ is spherical.
We turn to the first case. By Lemma \ref{B^0}(\ref{ij}), we have $\Hom_A(M,A)=0$.
Also, by Lemma \ref{tau_m}(\ref{proj}) and the Auslander--Gorenstein property of $\Y$, the projective $\Y$-module $\Y(-,\tau_mM)$ is injective, thus $\RHom_A(DA,\tau_mM)$ is a stalk complex by Lemma \ref{pr-in}(\ref{in}). Now, we have a complex
\[ \xymatrix{ \nu P_m\ar[r]& \nu P_{m-1}\ar[r]&\cdots\ar[r]& \nu P_1\ar[r]& \nu P_0 }, \]
which is an injective resolution of $\tau_mM$ by $\Tor_i^A(M,DA)=D\Ext^i_A(M,A)=0$ for $0\leq i\leq m-1$. Applying $\Hom_A(DA,-)$, we get a complex
\[ \RHom_A(DA,\tau_mM)=\left( \xymatrix{ P_m\ar[r]& P_{m-1}\ar[r]&\cdots\ar[r]& P_1\ar[r]& P_0 }\right).  \]
Now, this has a non-zero cohomology $M$ at the top, so for this to be a stalk complex, the map $P_m\to P_{m-1}$ has to be injective, in other words, $m=n$. This proves that $M$ is spherical.

We have seen that $\Y=\Y(A)$ being Auslander--Gorenstein implies $A$ is spherical. Since the Auslander--Gorenstein property is left-right symmetric, we see by working over $\Y^\op=\Y(A^\op)$ that $A^\op$ is also spherical. We conclude that $A$ is bispherical.

(\ref{resol})  The assertion for the case $M$ is projective is clear, so we assume that $M$ is non-projective. Since $M$ is spherical, the resolution in \eqref{resol} of $D\Y(M,-)$ simplifies to \eqref{close}. Moreover, as in Lemma \ref{B^0}, we may remove the identity summand on $(B^0X)_\ij$ from the first map of \eqref{close}, which results in the resolution
\begin{equation}\label{minimal}
	\xymatrix{ 0\ar[r]& \Y(-,(B^0X)_\ninj)\ar[r]&\Y(-,\nu P_0/(B^0X)_\ij)\oplus \Y(-,\tau_nM[n])\ar[r]& D\Y(M,-)\ar[r]& 0}.
\end{equation}
Since the first map is now a radical map, this resolution is minimal.

(\ref{ARbij})  By (\ref{resol}) it is enough to show that if $A$ is bispherical, then $B^0X$ is indecomposable non-injective whenever $M$ is indecomposable non-projective and $\Hom_A(M,A)\neq0$.
Let $P_1\to P_0\to M\to 0$ be the minimal projective presentation, so we have an exact sequence 
\[ \xymatrix{ 0\ar[r]&\tau M\ar[r]&\nu P_1\ar[r]& \nu P_0 }, \]
in which the second map is non-surjective by $\Hom_A(M,A)\neq0$, and has image $B^0X$. Since $A$ is cospherical and $\idim\tau M\geq2$, we have $\Ext^1_A(DA,\tau M)=0$, thus $B^0X=\Om^{-1}\tau M$ is indecomposable. By the minimality of $\nu P_1\to \nu P_0$, it has no injective summand. This proves the assertion.
\end{proof}

\section{Auslander--Reiten quivers of (co)spherical algebras}

Recall from \cref{thm::characterisationcospherical} that an algebra $A$ is cospherical if and only if the pair $(\T,\F)$ is a split torsion pair where $\T$ and $\F$ are defined by
\[
\begin{aligned}
\T&:=\{M\in\mod A\mid \Hom_A(M,A)=0\},\\
\F&:=\{M\in\mod A\mid M \text{ is a submodule of a projective module} \}.
\end{aligned}
\] 
The aim of this section is to show that this split torsion pair puts some heavy restrictions on the Auslander--Reiten quiver of $A$ if $A$ is representation-finite.
In fact, we show in Theorem \ref{theorem:directed} that such an algebra is directed.

\begin{lemma}\label{lemma:proj-inj}
    Every injective module in $\F$ is projective.
\end{lemma}

\begin{proof}
    Assume $I\in\F$ is injective. By definition we obtain a monomorphism $I\hookrightarrow P$ with $P$ projective. By definition of an injective object, this inclusion splits.
    % https://q.uiver.app/#q=WzAsMyxbMCwwLCJJIl0sWzEsMCwiUCJdLFswLDEsIkkiXSxbMCwxLCIiLDIseyJzdHlsZSI6eyJ0YWlsIjp7Im5hbWUiOiJob29rIiwic2lkZSI6InRvcCJ9fX1dLFswLDIsIlxcaWQiLDJdLFsxLDIsIiIsMix7InN0eWxlIjp7ImJvZHkiOnsibmFtZSI6ImRhc2hlZCJ9fX1dXQ==
    \[\begin{tikzcd}
    	I & P \\
    	I
    	\arrow[hook, from=1-1, to=1-2]
    	\arrow["\mathrm{id}"', from=1-1, to=2-1]
    	\arrow[dashed, from=1-2, to=2-1]
    \end{tikzcd}\]
    Hence $I$ is a direct summand of $P$ and therefore projective.
\end{proof}

\begin{definition}
Let $A$ be a finite dimensional algebra. A {\it path of length $n\geq 0$} in $\mod{A}$ is a finite sequence 
\[X=(X_0,X_1,\ldots,X_n)\]
of indecomposable $A$-modules $X_i$ for all $i$ such that $\rad(X_{i-1},X_i)\neq 0$ for $1\leq i\leq n$. If $X_0\cong X_n$ and $n \geq 1$ we call $X$ a cycle in $\mod{A}$. We say that $A$ is {\it directed} if there are no cycles in $\mod{A}$. 
For a representation-finite algebra $A$, being directed is equivalent to the Auslander--Reiten quiver $\AR(A)$ being acyclic, that is having no oriented cycles. 
\end{definition}

It is well known that directed algebras are representation-finite with acyclic quiver (and thus finite global dimension), see for example \cite[Chapter IX.]{ASS}. However, a representation-finite quiver algebra $KQ/I$ with $Q$ acyclic and admissible $I$ does not have to be directed as the next example shows.
\begin{example}
Let $A=KQ/I$ with 
$Q=$
% https://q.uiver.app/#q=WzAsMyxbMCwwLCIxIl0sWzEsMCwiMiJdLFsxLDEsIjMiXSxbMCwxLCJhIl0sWzEsMiwiYiJdLFswLDIsImMiLDJdXQ==
\[\begin{tikzcd}
	1 & 2 \\
	& 3
	\arrow["a", from=1-1, to=1-2]
	\arrow["c"', from=1-1, to=2-2]
	\arrow["b", from=1-2, to=2-2]
\end{tikzcd}\]
and $I=\langle ab \rangle.$
Then $A$ is representation-finite with $9$ indecomposable modules.
We have $\tau^2(S_2) \cong S_2$ and thus $A$ has no preprojective component and thus $A$ is not directed.
The algebra is neither spherical nor cospherical.
\end{example}

To prove Theorem \ref{theorem:directed}, we give the following intermediate step.
\begin{proposition}\label{theorem:tau-orbits}
    Assume $A$ is cospherical and representation-finite. Then there are no cyclic $\tau$-orbits.
\end{proposition}

\begin{proof}
    Assume there was a cyclic $\tau$-orbit in $\AR(A)$, that is a sequence 
    \[\tau^{\bullet}M=(M,\tau M,\ldots, \tau^n M,\ldots)\] 
    of indecomposable $A$-modules such that $M\cong \tau^nM$. We first observe that such a sequence would have to be completely contained in $\F$ or $\T$. Otherwise there would be a morphism from $\T$ to $\F$.
    % https://q.uiver.app/#q=WzAsMTEsWzcsMSwiTSJdLFs2LDAsIlhfMSJdLFs4LDAsIlhfMCJdLFs1LDEsIlxcdGF1XnsxfU0iXSxbOSwxLCJcXHRhdV57bi0xfU0iXSxbNCwwLCJYXzIiXSxbMywxLCJcXGNkb3RzIl0sWzIsMCwiWF97bi0xfSJdLFsxLDEsIlxcdGF1XntuLTF9TSJdLFsxMCwwLCJYX3tuLTF9Il0sWzAsMCwiWF8wIl0sWzAsMywiIiwwLHsic3R5bGUiOnsiYm9keSI6eyJuYW1lIjoiZGFzaGVkIn19fV0sWzQsMCwiIiwwLHsic3R5bGUiOnsiYm9keSI6eyJuYW1lIjoiZGFzaGVkIn19fV0sWzEsMF0sWzAsMl0sWzIsNF0sWzMsMV0sWzUsM10sWzYsNV0sWzgsN10sWzcsNl0sWzYsOCwiIiwwLHsic3R5bGUiOnsiYm9keSI6eyJuYW1lIjoiZGFzaGVkIn19fV0sWzMsNiwiIiwwLHsic3R5bGUiOnsiYm9keSI6eyJuYW1lIjoiZGFzaGVkIn19fV0sWzQsOV0sWzEwLDhdXQ==
    \[\begin{tikzcd}[column sep = small, row sep = small]
    	{X_0} && {X_{n-1}} && {X_2} && {X_1} && {X_0} && {X_{n-1}} \\
    	& {\tau^{n-1}M} && \cdots && {\tau^{1}M} && M && {\tau^{n-1}M}
    	\arrow[from=1-1, to=2-2]
    	\arrow[from=1-3, to=2-4]
    	\arrow[from=1-5, to=2-6]
    	\arrow[from=1-7, to=2-8]
    	\arrow[from=1-9, to=2-10]
    	\arrow[from=2-2, to=1-3]
    	\arrow[from=2-4, to=1-5]
    	\arrow[dashed, from=2-4, to=2-2]
    	\arrow[from=2-6, to=1-7]
    	\arrow[dashed, from=2-6, to=2-4]
    	\arrow[from=2-8, to=1-9]
    	\arrow[dashed, from=2-8, to=2-6]
    	\arrow[from=2-10, to=1-11]
    	\arrow[dashed, from=2-10, to=2-8]
    \end{tikzcd}\]

    Assume first that $\tau^{\bullet}M$ was contained in $\T$. Since projective covers exist in $\mod{A}$ and $A$ is representation-finite there exists a path $(P,X_1,X_2,\ldots ,X_{n-1},M)$ from an indecomposable projective $P$ to $M$ in $\AR(A)$. Since there is an arrow $X_{n-1}\to M$ in $\AR(A)$, there exists an arrow $\tau M\to X_{n-1}$. Thus $X_{n-1}$ can not be in $\F$ since $\Hom(\T,\F)=0$, in particular $X_{n-1}$ is not projective. It follows that there is an arrow $\tau X_{n-1}\to \tau M$ in $\AR(A)$. Continuing in this manner, we obtain the following picture in $\AR(A)$ which continues indefinitely.
    % https://q.uiver.app/#q=WzAsOCxbNywxLCJNIl0sWzYsMCwiWF97bi0xfSJdLFs1LDEsIlxcdGF1XnsxfU0iXSxbNCwwLCJcXHRhdSBYX3tuLTF9Il0sWzMsMSwiXFx0YXVeMiBNIl0sWzIsMCwiXFx0YXVeMiBYX3tuLTF9Il0sWzEsMSwiXFxsZG90cyJdLFswLDAsIlxcbGRvdHMiXSxbMCwyLCIiLDAseyJzdHlsZSI6eyJib2R5Ijp7Im5hbWUiOiJkYXNoZWQifX19XSxbMSwwXSxbMiwxXSxbMywyXSxbNCwzXSxbNiw1XSxbNSw0XSxbNCw2LCIiLDAseyJzdHlsZSI6eyJib2R5Ijp7Im5hbWUiOiJkYXNoZWQifX19XSxbMiw0LCIiLDAseyJzdHlsZSI6eyJib2R5Ijp7Im5hbWUiOiJkYXNoZWQifX19XSxbNyw2XSxbMSwzLCIiLDEseyJzdHlsZSI6eyJib2R5Ijp7Im5hbWUiOiJkYXNoZWQifX19XSxbMyw1LCIiLDEseyJzdHlsZSI6eyJib2R5Ijp7Im5hbWUiOiJkYXNoZWQifX19XSxbNSw3LCIiLDAseyJzdHlsZSI6eyJib2R5Ijp7Im5hbWUiOiJkYXNoZWQifX19XV0=
    \[\begin{tikzcd}[column sep = small, row sep = small]
    	\ldots && {\tau^2 X_{n-1}} && {\tau X_{n-1}} && {X_{n-1}} \\
    	& \ldots && {\tau^2 M} && {\tau^{1}M} && M
    	\arrow[from=1-1, to=2-2]
    	\arrow[dashed, from=1-3, to=1-1]
    	\arrow[from=1-3, to=2-4]
    	\arrow[dashed, from=1-5, to=1-3]
    	\arrow[from=1-5, to=2-6]
    	\arrow[dashed, from=1-7, to=1-5]
    	\arrow[from=1-7, to=2-8]
    	\arrow[from=2-2, to=1-3]
    	\arrow[from=2-4, to=1-5]
    	\arrow[dashed, from=2-4, to=2-2]
    	\arrow[from=2-6, to=1-7]
    	\arrow[dashed, from=2-6, to=2-4]
    	\arrow[dashed, from=2-8, to=2-6]
    \end{tikzcd}\]
    In particular the above argument shows $\tau^mX_{n-1}\in\T$ for all $m\geq 0$. Since $A$ is representation-finite, this means that the $\tau$-orbit $\tau^{\bullet}X_{n-1}$ is cyclic and completely contained in $\T$. Inductively this shows that $\tau^{\bullet}X_i\subseteq \T$ is cyclic for all $1\leq i\leq n-1$. However this implies that there exists an arrow $\tau X_1\to P$ in $\AR(A)$, a contradiction to the definition of $\T$. Thus there are no cyclic $\tau$-orbits in $\T$.

    Assume now that $\tau^{\bullet}M$ was contained in $\F$. Let $I$ be an indecomposable direct summand of the injective envelope $I_M$ of $M$. As above by representation-finiteness there exists a path $(M,Y_1,\ldots,Y_{n-1},I)$ in $\AR(A)$. First observe that, whenever $\tau^{-m}Y_1$ is defined it can not be in $\T$. Otherwise $\tau^{-m}Y_1\to\tau^{-(m+1)}M$ is a morphism from $\T$ to $\F$. Thus $\tau^{-m}Y_1\in\F$ if it is defined. If $\tau^{-m}Y_1$ is injective for some $m$, it is projective by \cref{lemma:proj-inj} which means already $Y_1$ was projective-injective. In that case we have the following Auslander--Reiten sequence for $Y_1$.
    \[\begin{tikzcd}[column sep = small, row sep = small]
    	& {Y_1} \\
    	{\rad(Y_1)} && {Y_1/\soc(Y_1)} \\
    	& {\rad(Y_1)/\soc(Y_1)}
    	\arrow[two heads, from=1-2, to=2-3]
    	\arrow[hook, from=2-1, to=1-2]
    	\arrow[two heads, from=2-1, to=3-2]
    	\arrow[dashed, from=2-3, to=2-1]
    	\arrow[hook, from=3-2, to=2-3]
    \end{tikzcd}\]
    
    %\[ \xymatrix@R=1mm@C=7mm{
    %&Y\ar@{->>}[dr]&\\
    %\rad(Y_1)\ar@{^(->}[ur]\ar@{->>}[dr]&& Y_1/\soc(Y_1)\\
    %&\rad(Y_1)/\soc(Y_1)\ar@{^(->}[ur]& } \]
    In particular $M=\rad(Y_1)$ and $Y_2=Y_1/\soc(Y_1)=\tau^{-1}M$ has a cyclic $\tau$-orbit. If $Y_1$ is not projective-injective, then by the above considerations $\tau^{-m}Y_1$ is defined for all $m\geq 0$ and representation-finiteness implies that $Y_1$ has a cyclic $\tau$-orbit. In both cases we can continue inductively to obtain that $I\in\F$. Since we can do this for every indecomposable summand of $I_M$, \cref{lemma:proj-inj} yields that every indecomposable module $M$ that has a cyclic $\tau$-orbit that is completely contained in $\F$ has a projective injective envelope. Moreover, the above argument yields that paths starting in these projective-injectives can only reach other projective-injectives or modules with cyclic $\tau$-orbits that are contained in $\F$, since they have to pass through the cyclic $\tau$-orbit of their radical. This implies that the cokernel of the injective envelope $\Omega^{-1}M\coloneqq\coker(M\hookrightarrow I_M)$ is a direct sum of projective-injective modules and modules with cyclic $\tau$-orbits, that is of the form 
    \[\Omega^{-1}M=\bigoplus_{i=1}^p I(i)\oplus \bigoplus_{j=1}^qC(j) \; ,\quad I(i) \text{ projective-injective}, \; \tau^{\bullet}C(j) \text{ cyclic}.\]
    If $q=0$, then $M$ is a direct summand of $I_M$ and hence $M$ would be projective-injective,  a contradiction to $\tau^{\bullet}M$ being cyclic. Thus $q\neq 0$. By induction it follows that $\idim M =\infty$, a contradiction to $\gldim A<\infty$.    
\end{proof}

One has the following example of an algebra whose AR-quiver has no cyclic $\tau$-orbits but the algebra is not directed. Therefore, one can not use \cref{theorem:tau-orbits} directly to prove \cref{theorem:directed}.
    %The following example shows that one can not use \cref{theorem:tau-orbits} directly to prove \cref{theorem:directed}. 
\begin{example}\label{example::nondiectedalgerbawithoutcyclictauorbit}
    %The following is an example of an acyclic representation-finite algebra of global dimension $3$ which is not directed, i.e. there is a cycle in the AR-quiver but there are no cyclic $\tau$-orbits in the AR-quiver. 
    Let $A=KQ/I$ for
    % https://q.uiver.app/#q=WzAsNCxbMSwwLCIxIl0sWzAsMSwiMiJdLFsyLDEsIjMiXSxbNCwxLCI0Il0sWzAsMSwiYSIsMl0sWzEsMiwiYiJdLFsyLDAsImMiLDJdLFsyLDMsImQiXV0=
    \[Q= \begin{tikzcd}[column sep = small, row sep = small]
    	& 1 \\
    	2 && 3 && 4
    	\arrow["a"', from=1-2, to=2-1]
    	\arrow["b", from=2-1, to=2-3]
    	\arrow["c"', from=2-3, to=1-2]
    	\arrow["d", from=2-3, to=2-5]
    \end{tikzcd}, \quad I=(bc,ca,bd).\]
    $A$ has global dimension $3$ and the AR-quiver looks as follows
    % https://q.uiver.app/#q=WzAsMTEsWzAsMSwiXFxwdXJwbGV7MX0iXSxbMSwyLCJcXHB1cnBsZXszXFxcXDFcXDs0fSJdLFswLDMsIlxccHVycGxlezR9Il0sWzIsMSwiM1xcXFw0Il0sWzIsMywiM1xcXFwxIl0sWzMsMiwiXFxwdXJwbGV7M30iXSxbNCwxLCJcXHB1cnBsZXsyXFxcXDN9Il0sWzUsMCwiXFxwdXJwbGV7MVxcXFwyXFxcXDN9Il0sWzYsMSwiMVxcXFwyIl0sWzUsMiwiXFxibHVlezJ9Il0sWzcsMiwiXFxwdXJwbGV7MX0iXSxbMCwxXSxbMSwzXSxbMyw1XSxbMSw0XSxbNCw1XSxbMiwxXSxbNSw2XSxbNiw5XSxbNiw3XSxbNyw4XSxbOCwxMF0sWzksOF0sWzMsMCwiIiwxLHsic3R5bGUiOnsiYm9keSI6eyJuYW1lIjoiZGFzaGVkIn19fV0sWzUsMSwiIiwxLHsic3R5bGUiOnsiYm9keSI6eyJuYW1lIjoiZGFzaGVkIn19fV0sWzQsMiwiIiwxLHsic3R5bGUiOnsiYm9keSI6eyJuYW1lIjoiZGFzaGVkIn19fV0sWzksNSwiIiwxLHsic3R5bGUiOnsiYm9keSI6eyJuYW1lIjoiZGFzaGVkIn19fV0sWzgsNiwiIiwxLHsic3R5bGUiOnsiYm9keSI6eyJuYW1lIjoiZGFzaGVkIn19fV0sWzEwLDksIiIsMSx7InN0eWxlIjp7ImJvZHkiOnsibmFtZSI6ImRhc2hlZCJ9fX1dXQ==
    
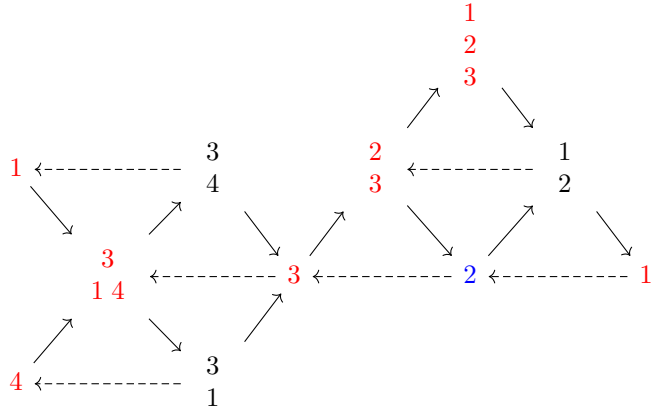
\[\begin{tikzcd}[column sep = small, row sep = small]
    	&&&&& \color{red}{\begin{array}{c} 1\\2\\3 \end{array}} \\
    	{\color{red}1} && \begin{array}{c} 3\\4 \end{array} && \color{red}{\begin{array}{c} 2\\3 \end{array}} && \begin{array}{c} 1\\2 \end{array} \\
    	& \color{red}{\begin{array}{c} 3\\1\;4 \end{array}} && \color{red}{3} && {\color{blue}2} && {\color{red}{1}} \\
    	{\color{red}{4}} && \begin{array}{c} 3\\1 \end{array}
    	\arrow[from=1-6, to=2-7]
    	\arrow[from=2-1, to=3-2]
    	\arrow[dashed, from=2-3, to=2-1]
    	\arrow[from=2-3, to=3-4]
    	\arrow[from=2-5, to=1-6]
    	\arrow[from=2-5, to=3-6]
    	\arrow[dashed, from=2-7, to=2-5]
    	\arrow[from=2-7, to=3-8]
    	\arrow[from=3-2, to=2-3]
    	\arrow[from=3-2, to=4-3]
    	\arrow[from=3-4, to=2-5]
    	\arrow[dashed, from=3-4, to=3-2]
    	\arrow[from=3-6, to=2-7]
    	\arrow[dashed, from=3-6, to=3-4]
    	\arrow[dashed, from=3-8, to=3-6]
    	\arrow[from=4-1, to=3-2]
    	\arrow[from=4-3, to=3-4]
    	\arrow[dashed, from=4-3, to=4-1]
    \end{tikzcd}\]
    Red vertices denote elements of $\mathcal{F}$, blue vertices denote elements of $\mathcal{T}$. Hence $A$ is not cospherical (i.e., not a counterexample to \cref{theorem:directed}).
\end{example}

\begin{theorem}\label{theorem:directed}
    Assume $A$ is cospherical and representation-finite. Then $A$ is directed. In particular, its Gabriel quiver does not contain a cycle.
\end{theorem}

Before we prove the above we first recall the following result, which is the most important ingredient for the proof.

\begin{theorem}[{\cite[Theorem~7]{BS}}]\label{theorem:BS}
    Let $A$ be an artin algebra. Then there are no sectional cycles in the AR-quiver of $A$, that is for every cycle $(C_0,\ldots,C_n)$ there exists $0\leq i\leq n-2$ such that $C_i\cong \tau C_{i+2}$.
\end{theorem}

\begin{proof}[Proof of \cref{theorem:directed}]
    Assume there was a cycle $C=(C_0,\ldots,C_n)$ in $\AR(A)$. We first observe that either $C\subseteq\T$ or $C\subseteq\F$, that is either all $C_i$ are in $\T$ or in $\F$. Otherwise there would be a morphism from $\T$ to $\F$, a contradiction.

    Consider first the case $C\subseteq \T$. In that case $C_i$ is not projective for all $0\leq i\leq n$ (otherwise $C_i\in\F$). Thus $\tau C_i$ is defined. Moreover there exists an arrow $C_i\to C_{i+1}$ in $\AR(A)$ if and only if there exists an arrow $\tau C_{i+1}\to C_i$ if and only if there exists an arrow $\tau C_{i}\to \tau C_{i+1}$.
    % https://q.uiver.app/#q=WzAsNCxbMiwxLCJDX2kiXSxbMywwLCJDX3tpKzF9Il0sWzEsMCwiXFx0YXUgQ197aSsxfSJdLFswLDEsIlxcdGF1IENfaSJdLFsyLDBdLFswLDFdLFszLDJdLFsxLDIsIiIsMix7InN0eWxlIjp7ImJvZHkiOnsibmFtZSI6ImRhc2hlZCJ9fX1dLFswLDMsIiIsMix7InN0eWxlIjp7ImJvZHkiOnsibmFtZSI6ImRhc2hlZCJ9fX1dXQ==
    \[\begin{tikzcd}[column sep = small, row sep = small]
    	& {\tau C_{i+1}} && {C_{i+1}} \\
    	{\tau C_i} && {C_i}
    	\arrow[from=1-2, to=2-3]
    	\arrow[dashed, from=1-4, to=1-2]
    	\arrow[from=2-1, to=1-2]
    	\arrow[from=2-3, to=1-4]
    	\arrow[dashed, from=2-3, to=2-1]
    \end{tikzcd}\]
    It follows that 
    \[\tau C=(\tau C_0,\ldots,\tau C_n)\]
    defines a cycle in $\AR(A)$. There are no cyclic $\tau$-orbits by \cref{theorem:tau-orbits} and hence, since $A$ is representation-finite, for every indecomposable $M\in\mod{A}$, there exists $0\leq p<\infty$ such that $\tau^pM$ is projective. Therefore there exists a minimal $m$ such that $\tau^mC$ contains an element of $\F$. As we observed above this implies that $\tau^m C\subseteq \F$. However by \cref{theorem:BS} there exists $0\leq i\leq n-2$ such that 
    \[\tau^mC_i\cong \tau^{m+1}C_{i+2}\] 
    or equivalently $\tau^{m-1}C_i\cong \tau^{m}C_{i+2}\in\F$. Hence $\tau^{m-1}C_{i}\in\F$, a contradiction to the minimality of $m$. Thus there does not exist a cycle in $\T$.

    Finally, we consider the case $C\subseteq\F$. The general idea is dual to the case $C\subseteq \T$, but we first might have to adjust $C$. Assume there is $C_j$ which is injective so that $\tau^{-1}C_j$ is not defined. Since $C_j\in\F$, by \cref{lemma:proj-inj} it is projective-injective. Thus we get the following Auslander--Reiten sequence for $C_j$.
    % https://q.uiver.app/#q=WzAsNCxbMCwxLCJcXHJhZChDX2opIl0sWzEsMCwiQ19qIl0sWzEsMiwiXFxyYWQoQ19qKS9cXHNvYyhDX2opIl0sWzIsMSwiQ19qL1xcc29jKENfaikiXSxbMSwzLCIiLDIseyJzdHlsZSI6eyJoZWFkIjp7Im5hbWUiOiJlcGkifX19XSxbMiwzLCIiLDAseyJzdHlsZSI6eyJ0YWlsIjp7Im5hbWUiOiJob29rIiwic2lkZSI6InRvcCJ9fX1dLFswLDEsIiIsMix7InN0eWxlIjp7InRhaWwiOnsibmFtZSI6Imhvb2siLCJzaWRlIjoidG9wIn19fV0sWzAsMiwiIiwwLHsic3R5bGUiOnsiaGVhZCI6eyJuYW1lIjoiZXBpIn19fV0sWzMsMCwiIiwxLHsic3R5bGUiOnsiYm9keSI6eyJuYW1lIjoiZGFzaGVkIn19fV1d
    \[\begin{tikzcd}[column sep = small, row sep = small]
    	& {C_j} \\
    	{\rad(C_j)} && {C_j/\soc(C_j)} \\
    	& {\rad(C_j)/\soc(C_j)}
    	\arrow[two heads, from=1-2, to=2-3]
    	\arrow[hook, from=2-1, to=1-2]
    	\arrow[two heads, from=2-1, to=3-2]
    	\arrow[dashed, from=2-3, to=2-1]
    	\arrow[hook, from=3-2, to=2-3]
    \end{tikzcd}\]
    Since $C_j$ is indecomposable projective-injective and occurs on an orriented cycle in $\AR(A)$, it is not simple. Hence $\rad(C_j)$ has simple socle and $C_j/\soc(C_j)$ has simple top and therefore both are indecomposable. Thus, it holds 
    \[C_{j-1}\cong\rad(C_j), \quad C_{j+1}\cong C_{j}/\soc(C_j).\] 
    Since $\F$ is closed under submodules and $\rad(C_{j})/\soc(C_j)\hookrightarrow C_{j}/\soc(C_j)$ is a monomorphism it holds $\rad(C_{j})/\soc(C_j)\in\F$. Take any indecomposable direct summand $D_j$ of $\rad(C_{j})/\soc(C_j)$. Then $D_j\in\F$ and moreover $D_j$ is not injective, otherwise it would be isomorphic to $C_j/\soc(C_j)$ since both modules are indecomposable. We define 
    \[C'=(C'_0,\ldots,C'_n)=\begin{cases}D_j & \text{if } C_j \text{ is injective,} \\ 
    C_j & \text{else}.
    \end{cases}\]
    By the above observations $C'$ defines a cycle in $\F$ that does not contain any injectives. Hence, by the dual argument from above $\tau^{-1}C'$ defines a cycle in $\AR(A)$. Observe that if $\tau^{-1}C'\subseteq\F$, then it does not contain any injective. Otherwise, by \cref{lemma:proj-inj}, it would be projective-injective and could not be $\tau^{-1}M$ for any $A$-module $M$. Again, since there are no cyclic $\tau$-orbits by \cref{theorem:tau-orbits} and $A$ is representation-finite there exists a minimal $1\leq m<\infty$ such that $\tau^{-m}(C')$ contains an element of $\T$. Then, by the above considerations, $\tau^{-m}C'\subseteq\T$ has to be completely contained in $\T$. However by \cref{theorem:BS} there exists $0\leq i\leq n-2$ such that 
    \[\tau^{-m}C'_{i}=\tau^{-m+1}C'_{i+2}.\] 
    Thus it holds $\tau^{-(m-1)}C'_{i+2}\in\T$, a contradiction to the minimality of $m$. It follows that such a cycle cannot exist and hence there is no cycle in $\F$.
\end{proof}

\begin{corollary}
    A (co)spherical Nakayama algebra is linear.
\end{corollary}

Combining this with Theorem \ref{thm::characterisationcospherical}, we get a complete characterization of (co)spherical Nakayama algebras in terms of their Kupisch series. Recall that a Nakayama algebra $A$ is uniquely determined by its Kupisch series $(c_1,\ldots,c_n)$, where $c_i$ is the dimension of the $i$-th indecomposable projective $A$-module, see \cite[Theorem~32.9]{AF}.

\begin{corollary}\label{coro:LinearNakayama}
    For a linear Nakayama algebra $A$ the following are equivalent.
    \begin{enumerate}[label = (\roman*)]
        \item $A$ is cospherical.
        %\item $(\T,\F)$ is a split torsion pair.
        %\item For every indecomposable module $M \in \mod A$ of injective dimension $\idim M \geq 2$ we have that $\Ext^1_A(DA,M)=0$.
        \item The top of every non-projective injective is the (unique) simple injective. 
        \item The Kupisch series is weakly decreasing.
    \end{enumerate}
    Moreover, any cospherical Nakayama algebra is of this kind.
\end{corollary}
\begin{proof}
Since a representation-finite cospherical algebra needs to be directed by Theorem \ref{theorem:directed}, a cospherical Nakayama algebra needs to be a linear Nakayama algebra.

%By Theorem \ref{thm::characterisationcospherical} we know that $A$ is cospherical if and only if for every indecomposable module $M \in \mod A$ of injective dimension $\idim M \geq 2$ we have that $\Ext^1_A(DA,M)=0$.

Since $A$ is a linear Nakayama algebra, we know that $\idim M \geq 2$ if and only if the top of the injective hull $I_M \hookleftarrow M$ is not injective.
Hence, if there is an indecomposable non-projective injective $I$ whose top is not injective, then $\idim \tau I \geq 2$ while $\Ext^1_{A}(I, \tau I) \neq 0$ so that $A$ cannot be cospherical.
The non-existence of such an injective is precisely the weakly decreasing property of the Kupisch series.

Finally, if the Kupisch series is weakly decreasing, one can see from the AR-quiver of $A$ that any indecomposable module either has a projective injective hull, and thus is in $\F$, or does not map to any projective and is thus in $\T$.

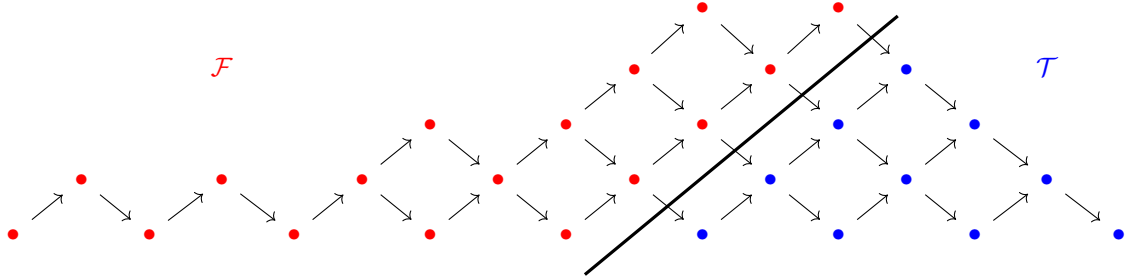
\vspace{-1cm}
\begin{equation*}
    \begin{tikzcd}[column sep = small, row sep = small]
        && && && && && && && \textcolor{white}{\bullet} \arrow[ddddddllllll, color = black, no head, very thick, ,shorten <=1cm]
        \\
        && && && && && \textcolor{red}{\bullet}\arrow[dr] && \textcolor{red}{\bullet}\arrow[dr] & 
        \\
        & && \textcolor{red}{\F} && && && \textcolor{red}{\bullet}\arrow[dr]\arrow[ur] && \textcolor{red}{\bullet}\arrow[dr]\arrow[ur] &&
        \textcolor{blue}{\bullet}\arrow[dr] && \textcolor{blue}{\T}
        \\
        && && && \textcolor{red}{\bullet}\arrow[dr] && \textcolor{red}{\bullet}\arrow[dr]\arrow[ur] && \textcolor{red}{\bullet}\arrow[dr]\arrow[ur] &&
        \textcolor{blue}{\bullet}\arrow[dr]\arrow[ur] && \textcolor{blue}{\bullet}\arrow[dr] && 
        \\
        & \textcolor{red}{\bullet}\arrow[dr] && \textcolor{red}{\bullet}\arrow[dr] && \textcolor{red}{\bullet}\arrow[dr]\arrow[ur] && \textcolor{red}{\bullet}\arrow[dr]\arrow[ur] && \textcolor{red}{\bullet}\arrow[dr]\arrow[ur] &&
        \textcolor{blue}{\bullet}\arrow[dr]\arrow[ur] && \textcolor{blue}{\bullet}\arrow[dr]\arrow[ur] && \textcolor{blue}{\bullet}\arrow[dr]
        \\
    	\textcolor{red}{\bullet}\arrow[ur] && \textcolor{red}{\bullet}\arrow[ur] && \textcolor{red}{\bullet}\arrow[ur] && \textcolor{red}{\bullet}\arrow[ur] && \textcolor{red}{\bullet}\arrow[ur] && 
        \textcolor{blue}{\bullet}\arrow[ur] && \textcolor{blue}{\bullet}\arrow[ur] && \textcolor{blue}{\bullet}\arrow[ur] && \textcolor{blue}{\bullet} \\
        && && && && \textcolor{white}{\bullet}
    \end{tikzcd}
\end{equation*}
Hence, $(\T, \F)$ is a split torsion pair and therefore, $A$ is cospherical by Theorem \ref{thm::characterisationcospherical}.
\end{proof}

\begin{remark}
    For a weakly decreasing Kupisch series we have $c_i-c_{i+1} \in \{0,1\}$ and $c_n=1, c_{n-1}=2$ and thus there are $2^{n-2}$ weakly decreasing Kupisch series for $n \geq 2$.
\end{remark}

\begin{corollary}\label{bisphericalNakayamaalgebras}
    For a Nakayama algebra $A$ the following are equivalent.
    \begin{enumerate}[label = (\roman*)]
        \item The Auslander--Yoneda algebra $Y(A)$ is Auslander--Gorenstein.
        \item $A$ is bispherical.
        \item $A$ is a truncated linear Nakayama algebra, i.e. $A = KQ/J^\ell$ for $Q$ the linearly oriented type $A$ quiver, $J$ the Jacobson radical and $\ell \geq 2$.
    \end{enumerate}
\end{corollary}

\begin{corollary}\label{cor::somelinearNakayamaalgebrasareqAG}
Let $A$ be a connected linear Nakayama algebra with weakly decreasing Kupisch series. Then $A$ is left quasi-Auslander--Gorenstein.
\end{corollary}
\begin{proof}
By Corollary \ref{coro:LinearNakayama}, cospherical linear Nakayama algebras are exactly those with weakly decreasing Kupisch series and by the dual of Proposition \ref{qAG} they are left quasi-Auslander--Gorenstein.
\end{proof}

\section{Replicated algebras of hereditary algebras}
For a finite dimensional algebra $A$, its {\it $m$-replicated algebra} is defined as the $(m+1)\times(m+1)$-matrix algebra
\[ A^{(m)}=
\begin{pmatrix} A&0& 0&\cdots& 0& 0 \\
DA&A&0&\cdots&0&0 \\
0&DA&A&\cdots&0&0 \\
\vdots&\vdots&\vdots&\ddots&\vdots&\vdots \\
0&0&0&\cdots&A&0\\
0&0&0&\cdots&DA&A
\end{pmatrix} \]

Notice in particular that the $m$-replication of the linear $A_n$-quiver is the Nakayama algebra $KQ/J^{n+1}$ for $Q$ the linearly oriented type $A_{(m+1)n}$ quiver and $J$ the Jacobson radical.
We have already seen that this algebra is always bispherical in Corollary \ref{bisphericalNakayamaalgebras}.

\begin{theorem}\label{replication}
Let $A$ be a hereditary algebra. Then its $m$-replication $A^{(m)}$ is bispherical for every $m\geq1$.
\end{theorem}
For basics on the representation theory of replicated algebras, we refer for example to \cite[Section 3]{CIM}.
We first observe that each indecomposable module over $A^{(m)}$ is a module over some $2\times 2$-part.
Consider, more generally, the matrix algebra 
\[ B=
\begin{pmatrix}
A_1&0& 0&\cdots& 0& 0 \\
M_1&A_2&0&\cdots&0&0 \\
0&M_2&A_3&\cdots&0&0 \\
\vdots&\vdots&\vdots&\ddots&\vdots&\vdots \\
0&0&0&\cdots&A_{n-1}&0\\
0&0&0&\cdots&M_{n-1}&A_n
\end{pmatrix}. \]
Then each $B$-module can be presented as a sequence $(X_1,\ldots,X_n; f_1,\ldots,f_{n-1})$ with $X_i\in\Mod A_i$ and $f_i\colon X_{i+1}\otimes_{A_{i+1}} M_i\to X_i$ in $\Mod A_i$ such that $f_{i}\circ(f_{i+1}\otimes1)=0$.
\begin{lemma}\label{2times2}
	Suppose for each $1<i<n$ that $A_i$ is hereditary and that $M_i$ is injective as a right $A_i$-module. If $(X_1,\ldots,X_n)$ presents an indecomposable $B$-module, then $X_i\neq0$ for at most two successive $i$. 
\end{lemma}
\begin{proof}
	We show that if there are three successive $i$ such that $X_i\neq0$, then the module is decomposable. We may assume that these are $i=1,2,3$, thus we have maps
	\[
	\begin{aligned}
		f_2&\colon X_3\otimes_{A_3}M_2\to X_2\\
		f_1&\colon X_2\otimes_{A_2}M_1\to X_1
	\end{aligned}
	\]
	such that the composite $X_3\otimes_{A_3}M_2\otimes_{A_2}M_1\to X_2\otimes_{A_2}M_1\to X_1$ is $0$.
	Since $M_2\in\inj A_2$, we have $X_3\otimes_{A_3}M_2\in\Fac M_2\subset \inj A_2$, and thus, $\Im f_2 \in \inj A_2$ since $A_2$ is hereditary. It gives rise to a decomposition $X_2=U\oplus V$ with $U=\Im f_2$, and thus a diagram
	\[ \xymatrix@R=1mm@!C=20mm{
		&U\otimes_{A_2}M_1\ar[dr]&\\
		X_3\otimes_{A_3}M_2\otimes_{A_2}M_1\ar@{->>}[ur]\ar[dr]_-0&\oplus &X_1\\
		&V\otimes_{A_2}M_1\ar[ur]& .} \]
	Now, this shows that $U\otimes_{A_2}M_1\to X_1$ must be $0$, and hence we have a decomposition $(X_1,X_2,X_3\ldots,X_n)=(X_1,V,0,\ldots,0)\oplus(0,U,X_3,\ldots,X_n)$.
\end{proof}

Now we consider $2\times 2$ triangular matrix algebras. 
Let $A$ and $B$ be rings, and let $M$ be an $(B,A)$-bimodule. Then a module over the triangular matrix algebra
\[ C=\begin{pmatrix} A&0\\M& B\end{pmatrix} \]
is given by $Z=(X,Y,f)$, where $X$ is an $A$-module, $Y$ is a $B$-module, and $f\colon Y\otimes_BM\to X$ is an $A$-linear map. We may equivalently give a $B$-linear map $g\colon Y\to\Hom_A(M,X)$ instead of $f$.
Let us compute the extensions $\Ext^i_C(Z,C)$ for $Z=(X,Y,f)$.
\begin{proposition}\label{Ext}
There exists a triangle in $\D(\Mod K)$:
\[ \xymatrix{\RHom_C(Z,C)\ar[r]^-{\left(\begin{smallmatrix}(-)e_1\\(-)e_2\end{smallmatrix}\right)}&\RHom_A(X,A\oplus M)\oplus \RHom_B(Y,B)\ar[r]^-{\left(\begin{smallmatrix}\cdot f&c\end{smallmatrix}\right)}&\RHom_A(Y\lotimes_BM,A\oplus M) }, \]
where $f\colon Y\lotimes_BM\to X$ is the multiplication map, and $c$ is induced by
\[ \xymatrix{ \RHom_B(Y,B)\ar[r]^-{\left(\begin{smallmatrix}0\\\lambda\end{smallmatrix}\right)}&\RHom_B(Y,\RHom_A(M,A)\oplus \REnd_A(M)) }, \]
where $\lambda\colon B\to\REnd_A(M)$ is the left multiplication, via the adjunction $\RHom_A(Y\lotimes_BM,A\oplus M)\simeq\RHom_B(Y,\RHom_A(M,A\oplus M))$.
\end{proposition}
\begin{proof}
	We denote by $e_1=\begin{pmatrix}1&0\\0&0\end{pmatrix}$, $e_2=\begin{pmatrix}0&0\\0&1\end{pmatrix}$ the idempotents in $C$.
	Consider the exact sequence
	\begin{equation}\label{C/(e_2)}
    \xymatrix{ 0\ar[r]& Ce_2\otimes_Be_2C\ar[r]& C\ar[r]&C/(e_2)\ar[r]& 0 }
    \end{equation}
	in $\Mod C^e$.
	Note that $Ce_2\otimes_Be_2C=Ce_2\lotimes_Be_2C$, so applying $Z\lotimes_C-$ yields a triangle
	\begin{equation}\label{Z}
    \xymatrix{ Y\lotimes_B{\begin{pmatrix}M&B\end{pmatrix}}\ar[r]& Z\ar[r]& Z\lotimes_CC/(e_2) }.
    \end{equation}
	This shows that, defining $N\in\rD(A)$ by the triangle $Y\lotimes_BM\to X\to N$, we have $Z\lotimes_CC/(e_2)=\begin{pmatrix}N&0\end{pmatrix}=N\lotimes_Ae_1C$.
	
	Next, we apply $\RHom_C(-,C)$ to the above triangle, which yields
	\[ \xymatrix{ \RHom_C(N\lotimes_Ae_1C,C)\ar[r]&\RHom_C(Z,C)\ar[r]&\RHom_B(Y,B)\ar[r]&\RHom_C(N[-1]\lotimes_Ae_1C,C) }, \]
	and the middle map identifies with the one induced by the functor $(-)e_2\colon\rD(C)\to\rD(B)$.
	
	Now, note that we have $\RHom_C(N\lotimes_Ae_1C,C)\xsimeq\RHom_A(N,Ce_1)$ along the functor $(-)e_1\colon\rD(C)\to\rD(A)$. It follows that the above triangle extends to an octahedron
	\begin{equation}\label{oct}
    \begin{gathered}\footnotesize
    \xymatrix@C=2.5mm{
		\RHom_C(N\lotimes_Ae_1C,C)\ar[r]\ar[d]^-\rsimeq&\RHom_C(Z,C)\ar[r]^-{(-)e_2}\ar[d]^-{(-)e_1}&\RHom_B(Y,B)\ar[r]\ar[d]&\RHom_C(N[-1]\lotimes_Ae_1C,C)\ar[d]^-\rsimeq\\
		\RHom_A(N,A\oplus M)\ar[r]&\RHom_A(X,A\oplus M)\ar[r]&\RHom_A(Y\lotimes_BM,A\oplus M)\ar[r]&\RHom_A(N[-1],A\oplus M), }
        \end{gathered}
    \end{equation}
	where the second row is given by the defining triangle of $N$.
    
    We claim that the third vertical map can be chosen so that collapsing the middle square yields the claimed triangle, from which the assertion follows.

    For this we work over resolutions. We take a projective resolution of $M$ as a $(B,A)$-bimodule, in particular as a right $A$-module, and regard $C$ as a dg algebra quasi-isomorphic to the original one.
    Also, we may assume that $Z$ is cofibrant as a dg $C$-module.
    Then applying $Z\otimes_C-$ to \eqref{C/(e_2)} gives a short exact sequence
    \begin{equation}\label{Zc}
    \xymatrix{ 0\ar[r]&Ze_2\otimes_Be_2C\ar[r]&Z\ar[r]&Z\otimes_CC/(e_2)\ar[r]& 0}
    \end{equation}
    which is a short exact sequence of cofibrant $C$-modules representing the triangle \eqref{Z}. Applying $(-)e_1$, we get a short exact sequence
    \begin{equation}\label{N}
    \xymatrix{ 0\ar[r]&Y\otimes_BM\ar[r]&X\ar[r]& \widetilde{N}\ar[r]& 0}
    \end{equation}
    consisting of cofibrant dg $A$-modules, and $\widetilde{N}$ represents $N$. In what follows, we will just denote $\widetilde{N}$ by $N$. Applying $\cHom_C(-,C)$ to \eqref{Zc} and $\cHom_A(-,Ce_1)$ to \eqref{N} yields a commutative diagram of short exact sequences
    \[ \xymatrix{
    0\ar[r]&\cHom_C(N\otimes_Ae_1C, C)\ar[r]\ar[d]^-\rsimeq&\cHom_C(Z,C)\ar[r]\ar[d]&\cHom_B(Y,B)\ar[r]\ar[d]&0\\
    0\ar[r]&\cHom_A(N,A\oplus M)\ar[r]&\cHom_A(X,A\oplus M)\ar[r]&\cHom_A(Y\otimes_BM,A\oplus M)\ar[r]& 0} \]
    which represents \eqref{oct}. Using this strictified sequence, we can directly obtain the claimed triangle from collapsing the strictified square.
\end{proof}

We note the following consequence which is used later.
\begin{corollary}\label{End}
In the setting of Proposition \ref{Ext}, suppose that $\lambda:B\xrightarrow{\simeq}\REnd_A(M)$ is a quasi-isomorphism. Then there is a triangle
\[ \xymatrix{\RHom_C(Z,C)\ar[r]^-{(-)e_1}&\RHom_A(X,A\oplus M)\ar[r]^-{\left(\begin{smallmatrix}\cdot f&0\end{smallmatrix}\right)}&\RHom_A(Y\lotimes_BM,A) }. \]
\end{corollary}
\begin{proof}
This follows from the description of the second map in Proposition \ref{Ext} since we may remove the isomorphism $B\xrightarrow{\simeq}\REnd_A(M)$ from the triangle.
\end{proof}

We note another easy consequence of Proposition \ref{Ext}.
\begin{lemma}\label{shita}
Let $C=\begin{pmatrix}A&0\\M&B\end{pmatrix}$ and assume that $M$ is injective as a right $A$-module.
If $X\in\mod A$ is spherical, then the $C$-module $Z=\begin{pmatrix}X&0\end{pmatrix}$ is spherical.
\end{lemma}
\begin{proof}
	Applying Proposition \ref{Ext} to $Y=0$, we get an isomorphism 
    \begin{equation*}
        \RHom_C(Z,C)\xsimeq\RHom_A(X,A)\oplus\RHom_A(X,M).
    \end{equation*}
    Now the assertion follows from $X$ being spherical and $\RHom_A(X,M)$ being concentrated in degree $0$ as $M$ is injective.
\end{proof}

Now we consider the case $A=B$ and $M=DA$, so that $C=\begin{pmatrix} A&0\\DA&A\end{pmatrix}$ is the first replication of $A$. We give a description of indecomposable objects over $C$.
\begin{lemma}\label{C-indec}
%Suppose that $A$ and $B$ are hereditary, and that $M$ is injective as a right $A$-module and as a left $B$-module. Then $Z=(X,Y,f)$ presents an indecomposable $C$-module if and only if we have one of the following.
Suppose that $A$ is hereditary. Then $Z=(X,Y,f)$ presents an indecomposable $C$-module if and only if we have one of the following.
\begin{enumerate}
\renewcommand\labelenumi{(\roman{enumi})}
\renewcommand\theenumi{\roman{enumi}}
\item $X$ is an indecomposable $A$-module and $Y=0$.
\item $Y$ is an indecomposable $A$-module and $X=0$.
\item $X$ is indecomposable injective, $Y$ is indecomposable projective, and $f\colon Y\otimes_ADA\to X$ is an isomorphism.
\item $0\neq X\in\inj A$, $0\neq Y\in\proj A$, and $f\colon Y\otimes_ADA\to X$ is a radical map which is surjective with indecomposable non-injective kernel. %and $g\colon Y\to\Hom_A(DA,X)$ is a radical map which is injective with indecomposable non-projective cokernel.
\end{enumerate}
\end{lemma}
\begin{proof}
	We prove that if $X\neq0$ and $Y\neq0$, then the conditions in (iii) or (iv) hold.
	
	Let us first show that $f\colon Y\otimes_ADA\to X$ is surjective. Since $A$ is hereditary, we have $Y\otimes_ADA\in\Fac DA=\inj A$, and the same holds for $\Im f$. It follows that $\Im f\to X$ is a split monomorphism. Its complement would yield a direct summand of $Z$ as a $C$-module, so we see that $f$ is surjective by indecomposability of $Z$. This also shows that $X$ is injective.
	%Dually, we see that $g\colon Y\to\Hom_A(DA,X)$ is injective and $Y$ is projective.
	
	Now, let $K=\Ker f$ so that we have an exact sequence $0\to K\to Y\otimes_ADA\xrightarrow{f}X\to 0$. 
    If $K=0$, then the case (iii) holds. Suppose $K\neq0$. Notice that $f$ gives an injective resolution of $K$, and since $Z\in\mod C$ is indecomposable, we see that $K$ is indecomposable and the resolution is minimal. %Similarly, we see that $g$ is the minimal projective resolution of $\Coker g$ which is indecomposable.
	Finally, $X\neq0$ implies that $\Ker f$ is non-injective. %and $Y\neq0$ implies $\Coker g$ is non-projective. 
    Therefore, the conditions in (iv) hold.

    We have seen that each indecomposable $C$-module $Z=(X,Y,f)$ satisfies either (i)--(iv). We prove the converse. It is clear that this is the case for $Z$ of type (i)--(iii). If $Z$ is given as in (iv), each direct summand must be again of the form (iv) since $f$ gives the minimal injective resolution of a non-injective module. We deduce that $Z$ must be indecomposable since $\Ker f$ is.
\end{proof}

Let us give the following result which serves as a first induction step for Theorem \ref{replication}.
\begin{proposition}\label{firstreplicationbispherical}
Suppose that $A$ is hereditary. Then its first replication $C=\begin{pmatrix} A&0\\ DA&A\end{pmatrix}$ is bispherical.
Moreover, for each indecomposable $A$-module $Y$ regarded  as a $C$-module $Z=\begin{pmatrix} 0&Y\end{pmatrix}$, the complex $\RHom_C(Z,C)$ is concentrated in a single degree.
\end{proposition}
\begin{proof}
	Let $Z$ be a $C$-module presented by $(X,Y,f)$ with $X,Y\in\mod A$ and $f\colon Y\otimes_ADA\to X$. We can compute $\Ext^i_C(Z,C)$ by the triangle in Proposition \ref{Ext}. By $A\xsimeq\End_A(DA)$, we may apply Corollary \ref{End} to get a triangle 
	\begin{equation}\label{sph?}
	\xymatrix{\RHom_C(Z,C)\ar[r]&\RHom_A(X,A\oplus DA)\ar[r]&\RHom_A(Y\lotimes_ADA,A) }.
	\end{equation}
	Suppose that $Z$ is indecomposable. By Lemma \ref{C-indec}, we have the following three cases.
	\begin{enumerate}
	\item $X\in\mod A$ is indecomposable and $Y=0$.
	\item $Y\in\mod A$ is indecomposable and $X=0$.
	\item $X\in\inj A$, $Y\in\proj A$, and there is an exact sequence $0\to K\to Y\otimes_ADA\xrightarrow{f} X\to 0$ with $K=0$ or $K$ is indecomposable non-injective.
	\end{enumerate}

	Suppose the case (1). Then we know that $Z=\begin{pmatrix} X&0\end{pmatrix}$ is spherical by Lemma \ref{shita}. %Then the triangle \eqref{sph?} shows $\RHom_C(Z,C)\xsimeq\RHom_A(X,A\oplus DA)$. Since $X$ is spherical, this certainly shows $Z$ is spherical.
	
	In the case (2), the triangle \eqref{sph?} yields an isomorphism $\RHom_C(Z,C)\ysimeq\RHom_A(Y\lotimes_ADA,A)[-1]$. Since $Y$ is indecomposable and $A$ is hereditary, the object $Y\lotimes_ADA\in\Db(\mod A)$ is a module or a shift of a module. It follows that $\RHom_A(Y\lotimes_ADA,A)[-1]$ is concentrated in one degree, and hence so is $\RHom_C(Z,C)$. In particular, $Z$ is spherical.
	
	Finally, assume the case (3). Noting that the second map of \eqref{sph?} is induced by $f\colon Y\lotimes_ADA\to X$, we get a triangle
	\[ \xymatrix{ \RHom_A(K,A)[-1]\ar[r]&\RHom_C(Z,C)\ar[r]&\RHom_A(X,DA) } \]
    by the octahedral axiom.
	Since $K\in\mod A$ is $0$ or indecomposable, the first term is concentrated in degree $\pdim K+1$. This shows that $Z$ is spherical.
	
	Note that if $\Si=\begin{pmatrix} \La&0\\ U&\Ga\end{pmatrix}$ for a $(\Ga,\La)$-bimodule $U$, then $\Si^\op=\begin{pmatrix} \Ga^\op&0\\ U&\La^\op\end{pmatrix}$ with $U$ viewed as a $(\La^\op,\Ga^\op)$-bimodule. We conclude that $C^\op=\begin{pmatrix} A^\op&0\\DA&A^\op\end{pmatrix}$ is also spherical, in other words, $C$ is bispherical. 
\end{proof}

\begin{proof}[Proof of Theorem \ref{replication}]
	We prove that the following hold by induction on $m$:
	\begin{enumerate}
    \renewcommand\labelenumi{(\roman{enumi})}
    \renewcommand\theenumi{\roman{enumi}}
	\item the $m$-replication is spherical.
	\item For each indecomposable $X=\begin{pmatrix}0&\cdots&0&Y\end{pmatrix}$ over $A^{(m)}$, the complex $\RHom_{A^{(m)}}(X,A^{(m)})$ is concentrated in a single degree.
	\end{enumerate}
    We have the claim for $m=1$ by Proposition \ref{firstreplicationbispherical}.

	Let $Z$ be an indecomposable $A^{(m)}$-module. Then by Lemma \ref{2times2}, we may decompose $A^{(m)}$ into blocks
	\[ A^{(m)}=
	\left( \begin{array}{ccc|cc|ccc} 
		A &&&&&&& \\ 
		DA&\ddots&&&&&& \\
		&\ddots&A&&&&& \\ 
		\hline
		&&DA&A&0&&& \\
		&&&DA&A&&&\\
		\hline
		&&&&DA&A&&\\
		&&&&&\ddots&\ddots&\\
		&&&&&&DA&A
	\end{array}\right)\]
in such a way that $Z$ is a module over the middle part. To show that $Z$ is spherical as a module over $A^{(m)}$, we may assume that $Z$ is a module over the lower right $2\times2$-part; otherwise, the induction hypothesis and Lemma \ref{shita} show that $Z$ is indeed spherical.

We are thus in the following setting: $A^{(m)}=\begin{pmatrix}A^{(m-2)}&0\\M&A^{(1)}\end{pmatrix}$ with $M=\begin{pmatrix}0&\cdots&0&DA\\ 0&\cdots&0&0\end{pmatrix}$, $Z=\begin{pmatrix}0&Y\end{pmatrix}$ presents an $A^{(m)}$-module, and $Y=\begin{pmatrix}Y_1&Y_2\end{pmatrix}$ presents an indecomposable module over $A^{(1)}=\begin{pmatrix} A&0\\DA&A\end{pmatrix}$.

Since $Y$ is an indecomposable $A^{(1)}$-module, by Lemma \ref{C-indec}, we have the following three cases.
\begin{enumerate}
	\item $Y_1\in\mod A$ is indecomposable and $Y_2=0$.
	\item $Y_2\in\mod A$ is indecomposable and $Y_1=0$.
	\item $0\neq Y_1\in\inj A$, $0\neq Y_2\in\proj A$, and there is an exact sequence $0\to K\to Y_2\otimes_ADA\xrightarrow{f} Y_1\to 0$ with $K=0$ or $K$ is indecomposable non-injective.
\end{enumerate}

In the case (1), we may view $Z$ as a module over $A^{(m-1)}$, so the induction hypothesis yields the result.

Suppose the case (2). Then we may decompose $A^{(m)}$ as
\[ A^{(m)}=\begin{pmatrix} A^{(m-1)}&0\\ N&A\end{pmatrix} \]
with $N=\begin{pmatrix}0&\cdots&0&DA\end{pmatrix}$.
Since the canonical map $A\to\REnd_{A^{(m-1)}}(N)$ is a quasi-isomorphism, applying Corollary \ref{End} to this decomposition yields 
\[ \xymatrix{ \RHom_{A^{(m)}}(Z,A^{(m)})\ar[r]^-\simeq&\RHom_{A^{(m-1)}}(Y_2\lotimes_AN,A^{(m-1)})[-1]}. \]
%Applying Proposition \ref{Ext} to this decomposition, we get a triangle
%\[ \xymatrix{ \RHom_{A^{(m)}}(Z,A^{(m)})\ar[r]&\RHom_A(Y_2,A)\ar[r]&\RHom_{A^{(m-1)}}(Y_2\lotimes_AN,A^{(m-1)}\oplus N)}. \]
%Since there is an isomorphism $A\xsimeq\REnd_{A^{(m-1)}}(N)$, this yields 
%$\RHom_{A^{(m)}}(Z,A^{(m)})\simeq\RHom_{A^{(m-1)}}(Y_2\lotimes_AN,A^{(m-1)})[-1]$. 
The same canonical isomorphism shows that the functor $-\lotimes_AN\colon\Db(\mod A)\to\Db(\mod A^{(m-1)})$ is fully faithful, so we see that $Y_2\lotimes_AN$ is indecomposable. It follows from the second induction hypothesis that $\RHom_{A^{(m-1)}}(Y_2\lotimes_AN,A^{(m-1)})$ is concentrated in a single degree. This shows that the same holds for $\RHom_{A^{(m)}}(Z,A^{(m)})$, completing the induction of (ii). In particular, $Z$ is spherical.

Finally, assume the case (3). Applying Proposition \ref{Ext} to the matrix $A^{(m)}=\begin{pmatrix}A^{(m-2)}&0\\M&A^{(1)}\end{pmatrix}$ and $Z=\begin{pmatrix}0&Y\end{pmatrix}$ as in the setting, we get a triangle
\[ \xymatrix{ \RHom_{A^{(m)}}(Z,A^{(m)})\ar[r]&\RHom_{A^{(1)}}(Y,A^{(1)})\ar[r]&\RHom_{A^{(m-2)}}(Y\lotimes_{A^{(1)}}M,A^{(m-2)}\oplus M) }. \]
Note that we have a surjection $A^{(1)}\to\End_{A^{(m-2)}}(M)=\REnd_{A^{(m-2)}}(M)$. Letting $e_1=\begin{pmatrix}1&0\\0&0\end{pmatrix}$ and $e_2=\begin{pmatrix}0&0\\0&1\end{pmatrix}$ be the idempotents in $A^{(1)}$, the kernel of this surjection is $A^{(1)}e_2\lotimes_Ae_2A^{(1)}$, which is isomorphic to $\begin{pmatrix}DA&A\end{pmatrix}=D(A^{(1)}e_1)$ as a right $A^{(1)}$-module, thus is injective. We therefore obtain an isomorphism
\[ \xymatrix{\RHom_{A^{(m)}}(Z,A^{(m)})\ar@{-}[r]^-\simeq&\RHom_{A^{(1)}}(Y,D(A^{(1)}e_1))\oplus \RHom_{A^{(m-2)}}(Y\lotimes_{A^{(1)}}M,A^{(m-2)})[-1] }. \]
Since $\RHom_{A^{(1)}}(Y,D(A^{(1)}e_1))$ is concentrated in degree $0$, it is enough to show that the second term of the right-hand side is a concentrated in a single degree.
For this we compute the projective resolution of $Y=\begin{pmatrix} Y_1&Y_2\end{pmatrix}$ over $A^{(1)}$. Recall that $Y_1\in\inj A$, $Y_2\in\proj A$, and there is an exact sequence
\[ \xymatrix{ 0\ar[r]&K\ar[r]&Y_2\otimes_ADA\ar[r]&Y_1\ar[r]& 0 } \]
with $K=0$ or indecomposable non-injective. This shows that there is an exact sequence
\[ \xymatrix@R=2mm{
0\ar[r]&P_1\otimes_Ae_1A^{(1)}\ar[r]&P_0\otimes_Ae_1A^{(1)}\ar[dr]\ar[rr]&&Y_2\otimes_Ae_2A^{(1)}\ar[r]&Y\ar[r]& 0 \\
	&&&{\begin{pmatrix} K&0\end{pmatrix}}\ar[ur]&&& } \]
%\[ \xymatrix{ 0\ar[r]&{\begin{pmatrix} K&0\end{pmatrix}}\ar[r]&Y_2\otimes_Ae_2A^{(1)}\ar[r]&Y\ar[r]& 0 } \]
in $\mod A^{(1)}$, which gives $Y\lotimes_{A^{(1)}}M=\begin{pmatrix}0&\cdots&0&K\lotimes_ADA\end{pmatrix}[1]$ as an $A^{(m-2)}$-module. Since $K$ is $0$ or indecomposable and $A$ is hereditary, this is a shift of an $A$-module located at the lower right entry of $A^{(m-2)}$. We conclude by induction hypothesis that $\RHom_{A^{(m-2)}}(Y\lotimes_{A^{(1)}}M,A^{(m-2)})[-1]$ is concentrated in a single degree.

%We conclude the proof by noting that $A^{(m)}$ is also cospherical. Note that if $\Si=\begin{pmatrix} \La&0\\ U&\Ga\end{pmatrix}$ for a $(\Ga,\La)$-bimodule $U$, then $\Si^\op=\begin{pmatrix} \Ga^\op&0\\ U&\La^\op\end{pmatrix}$ with $U$ viewed as an $(\La^\op,\Ga^\op)$-bimodule. This shows that 
Finally, notice as in the proof of Proposition \ref{firstreplicationbispherical} that the opposite of $A^{(m)}$ is $(A^\op)^{(m)}$.
Therefore, $A^{(m)}$ is bispherical.
\end{proof}

\begin{corollary}\label{cor::replciatedalgebraqAG}
Let $A$ be the $m$-th replicated algebra of a hereditary algebra. Then $A$ is quasi-Auslander--Gorenstein. 
\end{corollary}
\begin{proof}
This follows by combining the previous Theorem \ref{replication} with Proposition \ref{qAG}.
\end{proof}
This can be seen as some sort of generalization of the main result of \cite{CIM} for hereditary algebras, where it was shown that the $m$-th replicated algebras for certain values of $m$ are minimal Auslander--Gorenstein for Dynkin type algebras. 

The following example shows that even if $A$ is bispherical, the replication may not be spherical.
For the quiver and relations of the $m$-th replicated algebra for quiver algebras given by admissible relations consisting of monomial and commutativity relations, we follow the construction of \cite{Sch} for the repetitive algebra applied to only $m+1$ copies of the starting quiver.
\begin{example}
Let $\Lambda_1$ be the algebra presented by the following quiver with relations.
\[ \xymatrix{ 1\ar[r]^-a&2\ar[r]^-b& 3& ab=0} \]
Then its replication $\Lambda_2$ is presented by the following quiver with relations.
 \[ \xymatrix{
 	1\ar[r]^-a&2\ar[r]^-b\ar[d]_-{a^\ast}&3\ar[d]_-{b^\ast}&\\
 	&4\ar[r]^-a&5\ar[r]^-b& 6 }\qquad ab=0, \,  a^\ast a=b b^\ast \]
Then it is easy to see that the module $P(2)/\soc(P(2)) = \begin{smallmatrix} &2&\\3&&4\end{smallmatrix}$ is not spherical. This example also shows that a general triangular matrix ring $C=\begin{pmatrix} A&0\\ M&B\end{pmatrix}$ is not bispherical, even when the algebras $A$ and $B$ are bispherical and the $A$-$B$-bimodule $M$ is (co)spherical as an $A$-module and as a $B$-module. To see this note that in this example we have $\Lambda_2=\begin{pmatrix} \Lambda_1 &0\\ D(\Lambda_1) & \Lambda_1 \end{pmatrix}$ and $D(\Lambda_1)$ is (co)spherical as a left and right $\Lambda_1$-module.
\end{example}

\begin{example}
If $Q$ is a Dynkin quiver then the $m$-th replicated algebras of $KQ$ have especially interesting properties: They are always representation-finite and for infinitely many $m$ they are higher Auslander algebras, we refer to \cite[section 4]{CIM} for more details.
We just give here a simple example:
Let $Q$ be the Dynkin quiver of type $A_3$ with the following orientation:
% https://q.uiver.app/#q=WzAsMyxbMCwwLCIxIl0sWzEsMCwiMiJdLFsyLDAsIjMiXSxbMCwxLCJhIl0sWzIsMSwiYiIsMl1d
\[\begin{tikzcd}
	1 & 2 & 3
	\arrow["a", from=1-1, to=1-2]
	\arrow["b"', from=1-3, to=1-2]
\end{tikzcd}\]
Then the $m$-replicated algebra of $KQ$ for $m \geq 1$ can be described as follows:

Define a quiver $Q^{(m)}$ as follows.

\begin{enumerate}
\item The vertices are
\[
 i_r,
 \qquad i\in\{1,2,3\},\quad r=0,1,\dots,m.
\]
\item For every level $r=0,\dots,m$, there are horizontal arrows
\[
 1_r\xrightarrow{a_r}2_r\xleftarrow{b_r}3_r.
\]
\item For every adjacent pair of levels $r+1,r$, where $r=0,\dots,m-1$, there are returning arrows
\[
 \alpha_r:2_{r+1}\to1_r,
 \qquad
 \beta_r:2_{r+1}\to3_r.
\]
\end{enumerate}

And relations:
\[
 I^{(m)}=
 \left\langle
 a_{r+1}\beta_r,
 b_{r+1}\alpha_r,
 \alpha_r a_r-\beta_r b_r
 \;\middle|\;
 r=0,\dots,m-1
 \right\rangle.
\]
Then $(KQ)^{(m)} \cong K Q^{(m)}/I^{(m)}$ for $m \geq 1$.
For $m=2$ we have $Q^{(2)}=$
% https://q.uiver.app/#q=WzAsOSxbMCwwLCIxXzAiXSxbMSwwLCIyXzAiXSxbMiwwLCIzXzAiXSxbMCwxLCIxXzEiXSxbMCwyLCIxXzIiXSxbMSwxLCIyXzEiXSxbMSwyLCIyXzIiXSxbMiwxLCIzXzEiXSxbMiwyLCIzXzIiXSxbMCwxLCJhXzAiXSxbMiwxLCJiXzAiLDJdLFszLDUsImFfMSJdLFs3LDUsImJfMSIsMl0sWzQsNiwiYV8yIl0sWzgsNiwiYl8yIiwyXSxbNSwwLCJcXGFscGhhXzAiLDFdLFs1LDIsIlxcYmV0YV8wIiwxXSxbNiwzLCJcXGFscGhhXzEiLDFdLFs2LDcsIlxcYmV0YV8xIiwxXV0=
\[\begin{tikzcd}
	{1_0} & {2_0} & {3_0} \\
	{1_1} & {2_1} & {3_1} \\
	{1_2} & {2_2} & {3_2}
	\arrow["{a_0}", from=1-1, to=1-2]
	\arrow["{b_0}"', from=1-3, to=1-2]
	\arrow["{a_1}", from=2-1, to=2-2]
	\arrow["{\alpha_0}"{description}, from=2-2, to=1-1]
	\arrow["{\beta_0}"{description}, from=2-2, to=1-3]
	\arrow["{b_1}"', from=2-3, to=2-2]
	\arrow["{a_2}", from=3-1, to=3-2]
	\arrow["{\alpha_1}"{description}, from=3-2, to=2-1]
	\arrow["{\beta_1}"{description}, from=3-2, to=2-3]
	\arrow["{b_2}"', from=3-3, to=3-2]
\end{tikzcd}\]
and $I^{(2)}=\langle a_1 \beta_0, b_1 \alpha_0, \alpha_0 a_0-\beta_0 b_0, a_2 \beta_1, b_2 \alpha_1, \alpha_1 a_1-\beta_1 b_1 \rangle.$
By Theorem \ref{replication}, these algebras are bispherical.
\end{example}

\section{On a question by Venjakob}\label{section:7}
In his study of Iwasawa algebras \cite{V}, Venjakob considered certain quotient categories of a module category over an Auslander regular ring.
He raised a question about the structure of such categories.
We give a negative answer in general and provide a sufficient condition to have a positive answer in terms of our notion of (co)spherical algebras.

\subsection{The torsion pair}
To formulate the question of Venjakob, we first construct a torsion pair in a certain quotient category of $\mod\La$ for an Auslander regular algebra $\La$, more generally, for $\La$ satisfying the two-sided $(2,2)$-condition in the following sense.
\begin{definition}[\cite{Iy05c}]
Let $\La$ be a two-sided noetherian ring, and let
\[ 0\to \La\to I^0\to I^1\to\cdots \]
be the minimal injective resolution of the right $\La$-module $\La$.
For each $\ell,n>0$, the {\it $(\ell,n)$-condition} on $\La$ is defined by the following equivalent conditions.
\begin{itemize}
\item The flat dimension of $I^i$ is $<\ell$ for every $0\leq i<n$.
\item $\sgrade\Ext^\ell_{\La^\op}(X,\La)\geq n$ for all $X\in\mod\La^\op$.
\end{itemize}
\end{definition}
As before, the strong grade $\sgrade$ of a $\Lambda$-module $L$ is defined as 
\[
\sgrade_{\Lambda} L
:= \inf\{\, \operatorname{grade}_{\Lambda} M \mid M\subseteq L\,\}.
\]
It is easy to see that for each $n\geq0$, the subcategory formed by modules of strong grade $\geq n$ is a Serre subcategory.
We refer for example to \cite{I1} for more on the strong grade. We will use the following result in this section: 
\begin{lemma} \label{gradeequalstronggrade}
Let $\La$ be an Auslander--Gorenstein ring and $M$ a finitely generated $\Lambda$-module. Then $\grade M = \sgrade M$.
\end{lemma}
\begin{proof}
See for example \cite[Proposition 2.4]{I1}.
\end{proof}
Note that $\La$ is Auslander--Gorenstein if and only if $\idim\La<\infty$ on each side and satisfies the $(\ell,\ell)$-condition for every $\ell>0$.

We will be interested in the case where $\La$ satisfies the {\it two-sided $(2,2)$-condition}, that is, the $(2,2)$-condition on both $\La$ and on $\La^\op$.
By the left-right symmetry of the Auslander--Gorenstein property, see for example \cite[Theorem 3.7]{FGR}, Auslander-regular algebras form a subclass of rings satisfying the two-sided $(2,2)$-condition.

We set
\[ \refl\La=\{X\in\mod \La \mid X\to X^{**} \text{ is an isomorphism} \}, \]
the category of {\it reflexive modules}, where the map $X\to X^{**}$ is the evaluation map.
As pointed out in \cite{Ha9}, the two-sided $(2,2)$-condition is a sufficient condition for the category $\refl\La$ to behave well.

The following is the torsion pair we consider. We refer to \cite[Proposition 4.10]{BB} for the result for (commutative) normal surfaces.
\begin{proposition}\label{torp}
Let $\La$ be a two-sided noetherian ring satisfying the two-sided $(2,2)$-condition, and let
\begin{align*}
	\D&:=\{ X\in\mod\La \mid \sgrade X\geq2\},\\
	\E&:=\{ \Ext^1_{\La^\op}(X,\La)\in\mod\La \mid X\in\mod\La^\op\}.
	%\E&:=\{ X\in\mod\La \mid \sgrade X\geq1\}.
\end{align*}
Consider the Serre quotient $q\colon\mod\La\to(\mod\La)/\D$.
\begin{enumerate}
	\item The composite $\refl\La\to\mod\La\to(\mod\La)/\D$ is fully faithful.
	\item The pair $(q(\E),\refl\La)$ forms a torsion pair in $(\mod\La)/\D$.
%	\item The maximum exact structure on $\refl\La$ coincides with the extension-closed exact structure induced from $(\mod\La)/\D$.
\end{enumerate}
\end{proposition}
We need some preparation for the proof. In what follows, we write $\A:=(\mod\La)/\D$.
In particular, we recall the following classical result of \cite{AB}.

\begin{lemma}[{\cite[Proposition 2.6]{AB}}]\label{AB-sequence}
    Let $\La$ be a two-sided noetherian ring. Then for any $L\in\mod\La$ there is an exact sequence 
	\[ \xymatrix{ 0\ar[r]& \Ext^1_{\La^\op}(\Tr L,\La)\ar[r]&L\ar[r]&L^{**}\ar[r]&\Ext^2_{\La^\op}(\Tr L,\La)\ar[r]&0 } \]
    of $\La$-modules called the Auslander--Bridger sequence.
\end{lemma}

\begin{lemma}\label{ff}
	Let $L\in\mod\La$ and $M\in\refl\La$. Then the canonical map $\Hom_\La(L,M)\to\Hom_\A(L,M)$ is an isomorphism.
\end{lemma}
\begin{proof}
Note that for each $X\in\D$ and $M\in\refl\La$, we have $\Ext^i_\La(X,M)=0$ for $i=0,1$, see \cite[Lemma 3.3]{Ha9}.

	We first show injectivity. If a map $L\to M$ in $\mod\La$ is $0$ in $\A$, then it factors through an object $X\in\D$. Since we have $\Hom_\La(X,M)=0$ for $X\in\D$ and $M\in\refl\La$, we see that the original map $L\to M$ has to be $0$.
	
	We next prove surjectivity. Let a morphism $L\to M$ in $\A$ be presented by a diagram $L\xleftarrow{s} N\xrightarrow{f} M$ in $\mod\La$ with $\Ker s, \Coker s\in\D$. By $\Hom_\La(\Ker s,M)=0$, we see that $f$ factors through $\Im s$. Also by $\Ext^1_\La(\Coker s,M)=0$, we see that $f$ must factor through $L$.
\end{proof}

\begin{proof}[Proof of Proposition \ref{torp}]
	The assertion (1) is immediate from Lemma \ref{ff}, so we prove (2). If $T=\Ext^1_{\La^\op}(X,\La)$ for some $X\in\mod\La^\op$ and $M\in\refl\La$, then Lemma \ref{ff} shows that $\Hom_\A(T,M)=\Hom_\La(T,M)$, which is $0$ for example by \cite[Lemma 3.3]{Ha9}. Also, let $L\in\mod\La$ and consider the Auslander--Bridger sequence from Lemma \ref{AB-sequence} above.
	The last term has $\sgrade\geq2$ thus is $0$ in $\A$, the first term is in $\E$, and the third term is reflexive by \cite[Proposition 1.6, Theorem 1.7]{AR2}, see also \cite[Proposition 1.2]{Ha9}. Therefore it gives the required torsion sequence for a given object $L\in\A$.
\end{proof}

Finally, we note that in special cases the subcategory $\E$ can be more easily identified as it is quite closely related to the subcategory $\T$ we have considered before.

\begin{lemma}\label{lemma::T=E}
    Let 
    \begin{align*}
	\T&:=\{M\in\mod A\mid \Hom_A(M,A)=0\}
    \end{align*}
    and $\E$ as before. Then we have $\T \subseteq \E$.
    In particular, $\T = \E$ if and only if $\grade \Ext^1_{\La^\op}(X,\La) \geq 1$ for all $X\in\mod\La^\op$.
    \begin{proof}
        If $L \in \T$, then $L^{**} = 0$ and thus $L \cong \Ext^1_{\La^\op}(\Tr L,\La) \in \E$ by the Auslander--Bridger sequence from Lemma \ref{AB-sequence}.
    \end{proof}
\end{lemma}

\subsection{Venjakob's question}
The torsion pair given in the previous subsection leads to the following formulation of Venjakob's question, which was originally for Auslander regular algebras.
\begin{question}[{cf.\,\cite[Question 3.17]{V}}]
When is the torsion pair $(q(\E),\refl\La)$ in Proposition \ref{torp} split?
\end{question}
Note here that we need the $(2,2)$-condition for the subcategory ${\sgrade\geq2}$ to be a Serre subcategory so that the quotient makes sense. 

We next show that there is a positive answer for cospherical finite dimensional algebras that satisfy the two-sided (2,2)-condition.
\begin{proposition}
Let $A$ be a finite dimensional cospherical algebra satisfying the two-sided $(2,2)$-condition. Then the torsion pair $(q(\E),\refl A)$ in Proposition \ref{torp} is split.
\end{proposition}

\begin{proof}
Recall that $\D$ is the Serre subcategory of $\mod A$ consisting of modules of strong grade at least $2$. We have to show that each object in $(\mod A)/\D$ is isomorphic to a direct sum of objects in $q(\E)$ and in $\refl A$.

By Theorem \ref{thm::characterisationcospherical}, we have a split torsion pair $(\T,\F)$ in $\mod A$ with
\[
\begin{aligned}
\T&:=\{ M\in\mod A\mid \Hom_A(M,A)=0\}, \\
\F&:=\{ M\in\mod A\mid M \text{ is a submodule of a projective module}\}.
\end{aligned}
\]
Since $A^\op$ is spherical we have $\grade \Ext^1_{A^\op}(X,A) \geq 1$ for all $X\in\mod A^\op$ by Lemma \ref{lemma::sphericalcontrolsgradeofExt} and thus $\T=\E$ as subcategories of $\mod A$ by Lemma \ref{lemma::T=E}. Also, the Auslander--Bridger sequence, Lemma \ref{AB-sequence}, for $M\in\F$ shows that each object $M\in\F$ is isomorphic in $(\mod A)/\D$ to $M^{**}$ which is reflexive.
We conclude that $(\T,\F)$ being a split torsion pair implies that so is $(q(\E),\refl A)$.
\end{proof}

We immediately obtain the next result as any Auslander regular algebra satisfies the two-sided (2,2)-condition:
\begin{theorem} \label{splitARcorollary}
Let $A$ be a finite dimensional cospherical Auslander regular algebra. Then $(q(\E),\refl A)$ is split.
\end{theorem}

Motivated by the previous result, we checked whether the question of Venjakob still has a positive answer in the non-cospherical case for Auslander regular algebras. It turns out that the smallest (with respect to vector space dimension) Auslander regular Nakayama algebra that is not cospherical already gives a counterexample. So the question of Venjakob in the original formulation \cite[Question 3.17]{V} has a negative answer.
Due to its importance we state it as a Proposition rather than an example.
\begin{proposition} \label{Venjakobcounterexample}
Let $A$ be the Nakayama algebra given by the following quiver with relations
\[ \xymatrix{ 1\ar[r]^-\alpha&2\ar[r]^-\beta&3\ar[r]^-\gamma&4 &\alpha \beta=0}. \]
Then $A$ is an Auslander regular algebra such that the torsion pair $(q(\E),\refl A)$ in Proposition \ref{torp} is not split. In particular, this gives a negative answer to \cite[Question 3.17]{V}.
\end{proposition}

\begin{proof}
    Find below the AR-quiver of $A$. %and $A^{\mathrm{op}}$.
    Here, we label the modules by their support, that is the vertices $1 \leq i \leq 4$ to which a non-zero vector space is assigned. 
    
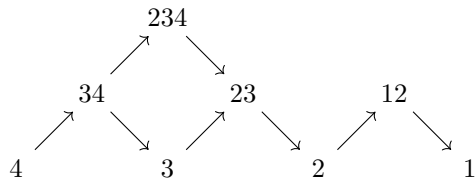
\begin{equation*}
        \begin{tikzpicture}
        \node at (-7,0) {4};\draw[->, shorten <= 1em , shorten >= 1em] (-7,0) -- (-6,1);
            \node at (-6,1) {34};\draw[->, shorten <= 1em , shorten >= 1em] (-6,1) -- (-5,2);\draw[->, shorten <= 1em , shorten >= 1em] (-6,1) -- (-5,0);
        \node at (-5,0) {3};\draw[->, shorten <= 1em , shorten >= 1em] (-5,0) -- (-4,1);
                \node at (-5,2) {234};\draw[->, shorten <= 1em , shorten >= 1em] (-5,2) -- (-4,1);
            \node at (-4,1) {23};\draw[->, shorten <= 1em , shorten >= 1em] (-4,1) -- (-3,0);
        \node at (-3,0) {2};\draw[->, shorten <= 1em , shorten >= 1em] (-3,0) -- (-2,1);
            \node at (-2,1) {12};\draw[->, shorten <= 1em , shorten >= 1em] (-2,1) -- (-1,0);
        \node at (-1,0) {1};
        %
        %\node at (1,0) {1};
        %    \node at (2,1) {21};
        %\node at (3,0) {2};
        %    \node at (4,1) {32};
        %\node at (5,0) {3};         \node at (5,2) {432};
        %    \node at (6,1) {43};
        %\node at (7,0) {4};
        \end{tikzpicture}
    \end{equation*}
    It is evident that $A$ is Auslander regular and that $1$ is the only indecomposable non-projective module on which $\Ext^1_A(-,A)$ vanishes. 
    Hence, we get that $\D = \add(1)$. 
    Therefore, $\mod(A)/\D \cong \mod(B)$ for $B$ the linear $A_3$-quiver since $2$ and $12$ become isomorphic. 

    Moreover, $A$ satisfies the $(1,1)$-condition as it is Auslander regular and thus we get 
    \begin{equation*}
        \E = \T = \{M\in\mod A\mid \Hom_A(M,A)=0\} = \add( 3 \oplus 1)
    \end{equation*}
    by Lemma \ref{lemma::T=E}. 
    Thus, $q(23) \notin q(\E) = \add (q(3))$. 
    
    On the other hand, we see that $\Hom_A(12, A) \cong \Hom_A(2, A) \cong \Hom_A(23, A)$ as $A^\op$-modules via precomposition so that $\refl A = \proj A$ and in particular, $23$ is neither reflexive, nor isomorphic to a reflexive module in $\mod(A)/\D$.
    Hence, as $q(23)$ remains indecomposable, we get that $q(23) \notin q(E) \oplus \refl A \neq \mod(A)/\D$.
\end{proof}

\section{Further examples and counterexamples}
In this section we give several examples and counterexamples.
Our first example is a bispherical algebra that is neither a replicated algebra of a hereditary algebra nor a Nakayama algebra:
\begin{example}
Let $B$ be the incidence algebra of a Boolean lattice of a 2-set, given by quiver
% https://q.uiver.app/#q=WzAsNCxbMCwxLCIxIl0sWzEsMCwiMiJdLFsyLDEsIjQiXSxbMSwyLCIzIl0sWzAsMSwiXFxhbHBoYSJdLFsxLDIsIlxcYmV0YSJdLFswLDMsIlxcZ2FtbWEiLDJdLFszLDIsIlxcZGVsdGEiLDJdXQ==
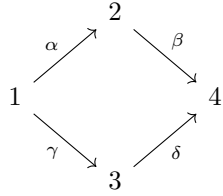
\[\begin{tikzcd}
	& 2 & \\
	1 && 4 \\
	& 3
	\arrow["\beta", from=1-2, to=2-3]
	\arrow["\alpha", from=2-1, to=1-2]
	\arrow["\gamma"', from=2-1, to=3-2]
	\arrow["\delta"', from=3-2, to=2-3]
\end{tikzcd}\] and relations $\langle \alpha \beta-\gamma \delta \rangle $. 
$B$ is an algebra of global dimension 2 with 11 indecomposable modules. %We leave it to the reader to verify that all 11 indecomposable modules are spherical and since $B \cong B^{op}$, $B$ is bispherical.
As indicated in the picture below, $(\textcolor{blue}{\T},\textcolor{red}{\F})$ defines a split torsion-pair on $\mod(B)$, hence $B$ is cospherical and since $B \cong B^{op}$, $B$ is bispherical.
We denote by $P_i$ the indecomposable projective $B$-modules, by $I_i$ the indecomposable injective $B$-modules and by $S_i$ the simple $B$-modules.
We write $M=\rad P_1$ and $N=P_1/\soc P_1$. 
%$M$ is the unique indecomposable $B$-module with dimension vector $[0,1,1,1]$ and $N$ is the unique indecomposable $B$-module with dimension vector $[1,1,1,0]$.
The Auslander--Yoneda algebra of $B$ is given by $(KQ/I)^{op}$ (which is isomorphic to $KQ/I$ as $B \cong B^{op}$) with $Q$ and $I$ as follows: 
$Q=$
% https://q.uiver.app/#q=WzAsMTEsWzAsMSwiU180Il0sWzEsMCwiUF8yIl0sWzEsMiwiUF8zIl0sWzIsMSwiTSJdLFszLDAsIlNfMiJdLFszLDEsIlBfMSJdLFszLDIsInNfMyJdLFs0LDEsIk4iXSxbNSwwLCJJXzIiXSxbNSwyLCJJXzMiXSxbNiwxLCJTXzEiXSxbMCwxLCJjIl0sWzAsMiwiZCIsMl0sWzEsMywiZyIsMl0sWzIsMywiaCJdLFszLDQsInAiLDJdLFszLDUsInIiLDJdLFszLDYsInEiLDJdLFs2LDcsImIiLDJdLFs1LDcsInUiLDFdLFs0LDcsImEiLDFdLFs3LDgsInMiLDFdLFs3LDksInQiLDFdLFs5LDEwLCJmIiwxXSxbOCwxMCwiZSIsMV0sWzEwLDMsIngiLDEseyJsYWJlbF9wb3NpdGlvbiI6MjAsIm9mZnNldCI6MywiY29sb3VyIjpbMCw2MCw2MF19LFswLDYwLDYwLDFdXSxbOCwyLCJ5IiwxLHsibGFiZWxfcG9zaXRpb24iOjgwLCJvZmZzZXQiOi00LCJjb2xvdXIiOlswLDYwLDYwXX0sWzAsNjAsNjAsMV1dLFs5LDEsInoiLDEseyJsYWJlbF9wb3NpdGlvbiI6ODAsImNvbG91ciI6WzAsNjAsNjBdfSxbMCw2MCw2MCwxXV0sWzcsMCwidyIsMix7ImxhYmVsX3Bvc2l0aW9uIjo4MCwib2Zmc2V0IjotMywiY29sb3VyIjpbMCw2MCw2MF19LFswLDYwLDYwLDFdXV0=
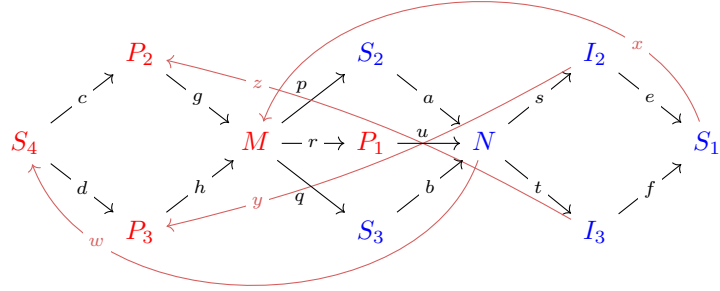
\[\begin{tikzcd}
	& {\textcolor{red}{P_2}} && {\textcolor{blue}{S_2}} && {\textcolor{blue}{I_2}} & \\
	{\textcolor{red}{S_4}} && {\textcolor{red}{M}} & {\textcolor{red}{P_1}} & {\textcolor{blue}{N}} && {\textcolor{blue}{S_1}} \\
	& {\textcolor{red}{P_3}} && {\textcolor{blue}{S_3}} && {\textcolor{blue}{I_3}}
	\arrow["g"{description}, from=1-2, to=2-3]
	\arrow["a"{description}, from=1-4, to=2-5]
	\arrow["e"{description}, from=1-6, to=2-7]
	\arrow["s"{description}, from=2-5, to=1-6]
	\arrow["y"{description, pos=0.8}, color={rgb,255:red,214;green,92;blue,92}, from=1-6, to=3-2, bend left=8]
	\arrow["c"{description}, from=2-1, to=1-2]
	\arrow["d"{description}, from=2-1, to=3-2]
	\arrow["p", from=2-3, to=1-4]
	\arrow["r"{description}, from=2-3, to=2-4]
	\arrow["q"', from=2-3, to=3-4]
	\arrow["w"{description, pos=0.8}, color={rgb,255:red,214;green,92;blue,92}, from=2-5, to=2-1, bend left=70]
	\arrow["t"{description}, from=2-5, to=3-6]
	\arrow["x"{description, pos=0.2}, color={rgb,255:red,214;green,92;blue,92}, from=2-7, to=2-3,bend right=70]
	\arrow["h"{description}, from=3-2, to=2-3]
	\arrow["z"{description, pos=0.8}, color={rgb,255:red,214;green,92;blue,92}, from=3-6, to=1-2, bend right = 8]
	\arrow["b"{description}, from=3-4, to=2-5]
	\arrow["f"{description}, from=3-6, to=2-7]
    \arrow["u"{pos=0.4}, from=2-4, to=2-5]
\end{tikzcd}\]
with relations $I=$

$\langle ex-yh,
wc-tz,
wd-sy,
se-tf,
fx-zg,
ru-pa-qb,
gq,
hp,
uw,
xr,
at,
bs,
cg-dh \rangle$.
\end{example}
We remark that the previous example is isomorphic to the tensor product of the hereditary algebra of Dynkin type $A_2$ with itself.
However, the next example shows that it is in general not true that being a spherical algebra is closed under tensor products.
\begin{example}
Let $B_1=K A_2$ be the hereditary linear oriented path algebra of Dynkin type $A_2$ and $B_2=KA_3$ the hereditary linear oriented path algebra of Dynkin type $A_3$ and let $A:= B_1 \otimes B_2$.
$A$ is given by quiver and relations as follows $A=KQ/I$ with 
$Q=$
% https://q.uiver.app/#q=WzAsNixbMCwwLCIxIl0sWzEsMCwiMiJdLFswLDEsIjMiXSxbMSwxLCI0Il0sWzAsMiwiNSJdLFsxLDIsIjYiXSxbMCwxLCJhIl0sWzAsMiwiYyIsMl0sWzIsMywiZCIsMl0sWzEsMywiYiJdLFsyLDQsImUiXSxbNCw1LCJmIl0sWzMsNSwiZyJdXQ==
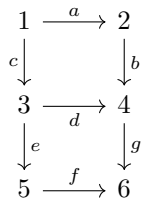
\[\begin{tikzcd}
	1 & 2 \\
	3 & 4 \\
	5 & 6
	\arrow["a", from=1-1, to=1-2]
	\arrow["c"', from=1-1, to=2-1]
	\arrow["b", from=1-2, to=2-2]
	\arrow["d"', from=2-1, to=2-2]
	\arrow["e", from=2-1, to=3-1]
	\arrow["g", from=2-2, to=3-2]
	\arrow["f", from=3-1, to=3-2]
\end{tikzcd}\]
and $I=\langle ab-cd, dg-ef \rangle$.
The unique indecomposable module with dimension vector $[1,1,1,0,0,0]$ has projective dimension 2 and is not spherical.
\end{example}

We next give the Auslander--Yoneda algebra of a bispherical Nakayama algebra.
\begin{example}
The Auslander--Yoneda algebra of the Nakayama algebra with Kupisch series [4,4,3,2,1] is given as $A=KQ/I$ with quiver
$Q=$
% https://q.uiver.app/#q=WzAsMTQsWzQsMywiNyJdLFszLDIsIjUiXSxbMiwzLCIzIl0sWzEsMiwiMiJdLFswLDMsIjEiXSxbMiwxLCI0Il0sWzQsMSwiOCJdLFs1LDIsIjkiXSxbNiwzLCIxMSJdLFszLDAsIjYiXSxbNSwwLCIxMCJdLFs2LDEsIjEyIl0sWzcsMiwiMTMiXSxbOCwzLCIxNCJdLFsxLDAsImFfezU3fSIsMV0sWzIsMSwiYV97MzV9IiwxXSxbMywyLCJhX3syM30iLDFdLFs0LDMsImFfezEyfSIsMV0sWzMsNSwiYV97MjR9IiwxXSxbNSwxLCJhX3s0NX0iLDFdLFsxLDYsImFfezU4fSIsMV0sWzYsNywiYV97ODl9IiwxXSxbMCw3LCJhX3s3OX0iLDFdLFs3LDgsImFfezksMTF9IiwxXSxbNSw5LCJhX3s0Nn0iLDFdLFs5LDYsImFfezY4fSIsMV0sWzYsMTAsImFfezgsMTB9IiwxXSxbMTAsMTEsImFfezEwLDEyfSIsMV0sWzcsMTEsImFfezksMTJ9IiwxXSxbMTEsMTIsImFfezEyLDEzfSIsMV0sWzgsMTIsImFfezExLDEzfSIsMV0sWzEyLDEzLCJhX3sxMywxNH0iLDFdLFsxMyw2LCJ0X3sxNCw4fSIsMSx7ImNvbG91ciI6WzAsNjAsNjBdfSxbMCw2MCw2MCwxXV0sWzEyLDEsInRfezEzLDV9IiwxLHsibGFiZWxfcG9zaXRpb24iOjMwLCJjdXJ2ZSI6LTIsImNvbG91ciI6WzAsNjAsNjBdfSxbMCw2MCw2MCwxXV0sWzcsMywidF97OTJ9IiwxLHsibGFiZWxfcG9zaXRpb24iOjcwLCJjdXJ2ZSI6LTIsImNvbG91ciI6WzAsNjAsNjBdfSxbMCw2MCw2MCwxXV0sWzgsNSwidF97MTEsNH0iLDEseyJsYWJlbF9wb3NpdGlvbiI6NjAsImNvbG91ciI6WzAsNjAsNjBdfSxbMCw2MCw2MCwxXV0sWzExLDIsInRfezEyLDN9IiwxLHsibGFiZWxfcG9zaXRpb24iOjQwLCJjb2xvdXIiOlswLDYwLDYwXX0sWzAsNjAsNjAsMV1dLFs2LDQsInRfezgxfSIsMSx7ImNvbG91ciI6WzAsNjAsNjBdfSxbMCw2MCw2MCwxXV1d
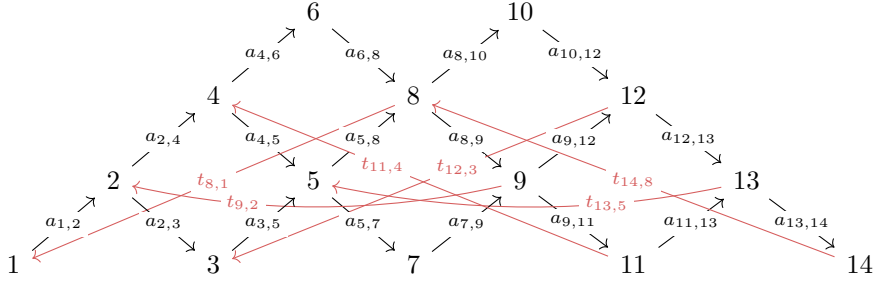
\[\begin{tikzcd}
	&&& 6 && 10 \\
	&& 4 && 8 && 12 \\
	& 2 && 5 && 9 && 13 \\
	1 && 3 && 7 && 11 && 14
	\arrow["{a_{6,8}}"{description}, from=1-4, to=2-5]
	\arrow["{a_{10,12}}"{description}, from=1-6, to=2-7]
	\arrow["{a_{4,6}}"{description}, from=2-3, to=1-4]
	\arrow["{a_{4,5}}"{description}, from=2-3, to=3-4]
	\arrow["{a_{8,10}}"{description}, from=2-5, to=1-6]
	\arrow["{a_{8,9}}"{description}, from=2-5, to=3-6]
	\arrow["{t_{8,1}}"{description}, color={rgb,255:red,214;green,92;blue,92}, from=2-5, to=4-1]
	\arrow["{a_{12,13}}"{description}, from=2-7, to=3-8]
	\arrow["{t_{12,3}}"{description, pos=0.4}, color={rgb,255:red,214;green,92;blue,92}, from=2-7, to=4-3]
	\arrow["{a_{2,4}}"{description}, from=3-2, to=2-3]
	\arrow["{a_{2,3}}"{description}, from=3-2, to=4-3]
	\arrow["{a_{5,8}}"{description}, from=3-4, to=2-5]
	\arrow["{a_{5,7}}"{description}, from=3-4, to=4-5]
	\arrow["{a_{9,12}}"{description}, from=3-6, to=2-7]
	\arrow["{t_{9,2}}"{description, pos=0.7}, color={rgb,255:red,214;green,92;blue,92}, curve={height=-12pt}, from=3-6, to=3-2]
	\arrow["{a_{9,11}}"{description}, from=3-6, to=4-7]
	\arrow["{t_{13,5}}"{description, pos=0.3}, color={rgb,255:red,214;green,92;blue,92}, curve={height=-12pt}, from=3-8, to=3-4]
	\arrow["{a_{13,14}}"{description}, from=3-8, to=4-9]
	\arrow["{a_{1,2}}"{description}, from=4-1, to=3-2]
	\arrow["{a_{3,5}}"{description}, from=4-3, to=3-4]
	\arrow["{a_{7,9}}"{description}, from=4-5, to=3-6]
	\arrow["{t_{11,4}}"{description, pos=0.6}, color={rgb,255:red,214;green,92;blue,92}, from=4-7, to=2-3]
	\arrow["{a_{11,13}}"{description}, from=4-7, to=3-8]
	\arrow["{t_{14,8}}"{description}, color={rgb,255:red,214;green,92;blue,92}, from=4-9, to=2-5]
\end{tikzcd}\]
and $I$ the ideal containing all commutativity relations and additionally the following zero relations.
\begin{equation*}
    \{a_{1,2}a_{2,3}, a_{3,5}a_{5,7}, a_{7,9}a_{9,11}, a_{11,13}a_{13,14},a_{6,8}t_{8,1}, t_{11,4}a_{4,6}, a_{10,12}t_{12,3}, t_{14,8}a_{8,10}\} \subset  I
\end{equation*}

Here, we view the arrows $t$ as the extensions with projective-injective middle term and maps canonically given by the associated paths in our quiver. 
(Notice that you may view any rectangle in the AR-quiver as a short exact sequence by equipping the diagonally up-going path to the cokernel with a $-1$.)
The additional zero relations not coming from the AR quiver then just realize the fact that projective injectives have no non-trivial extensions.

\end{example}
We next give an example of a Nakayama algebra that is spherical but not cospherical and it will also show that being a spherical algebra is not closed under derived equivalences:
\begin{example}
The Auslander--Yoneda algebra of the Nakayama algebra with Kupisch series [2,3,2,1] is given as $A=(KQ/I)^{\mathrm{op}}$ with quiver
$Q=$

% https://q.uiver.app/#q=WzAsOCxbMCwyLCIxIl0sWzEsMSwiMiJdLFsyLDAsIjMiXSxbMywxLCI1Il0sWzIsMiwiNCJdLFs0LDIsIjYiXSxbNSwxLCI3Il0sWzYsMiwiOCJdLFswLDEsImFfMSJdLFsxLDIsImFfMiJdLFsyLDMsImFfNCJdLFsxLDQsImFfMyJdLFs0LDMsImFfNSIsMl0sWzMsNSwiYV82Il0sWzUsNiwiYV83Il0sWzYsNywiYV84Il0sWzcsNSwidF97MX0iLDEseyJjb2xvdXIiOlswLDYwLDYwXX0sWzAsNjAsNjAsMV1dLFs1LDEsInRfMiIsMSx7ImxhYmVsX3Bvc2l0aW9uIjoyMCwiY29sb3VyIjpbMCw2MCw2MF19LFswLDYwLDYwLDFdXSxbMywwLCJ0XzMiLDEseyJsYWJlbF9wb3NpdGlvbiI6NzAsImNvbG91ciI6WzAsNjAsNjBdfSxbMCw2MCw2MCwxXV1d
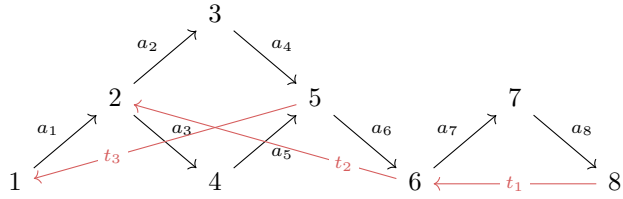
\[\begin{tikzcd}
	&& 3 \\
	& 2 && 5 && 7 \\
	1 && 4 && 6 && 8
	\arrow["{a_4}", from=1-3, to=2-4]
	\arrow["{a_2}", from=2-2, to=1-3]
	\arrow["{a_3}", from=2-2, to=3-3]
	\arrow["{t_3}"{description, pos=0.7}, color={rgb,255:red,214;green,92;blue,92}, from=2-4, to=3-1]
	\arrow["{a_6}", from=2-4, to=3-5]
	\arrow["{a_8}", from=2-6, to=3-7]
	\arrow["{a_1}", from=3-1, to=2-2]
	\arrow["{a_5}"', from=3-3, to=2-4]
	\arrow["{t_2}"{description, pos=0.2}, color={rgb,255:red,214;green,92;blue,92}, from=3-5, to=2-2]
	\arrow["{a_7}", from=3-5, to=2-6]
	\arrow["{t_{1}}"{description}, color={rgb,255:red,214;green,92;blue,92}, from=3-7, to=3-5]
\end{tikzcd}\]
and relations $I=\langle a_1 a_3,a_5 a_6,a_7 a_8,a_2 a_4-a_3 a_5,t_2 a_2,a_4 t_3,a_6 t_2-t_3 a_1,t_1 a_7,a_8 t_1 \rangle.$
The opposite is due to our convention of composition of arrows being opposite to that of morphisms -- which is the canonical composition order for the Auslander--Yoneda algebra.

We also remark that as a gentle tree algebra this algebra is derived equivalent to a hereditary algebra of Dynkin type $A_4$ (see for example \cite[Chapter IX. 6]{ASS}), so being bispherical is not preserved under derived equivalences.
\end{example}

Our last observation shows that the condition that all simple left and right modules are spherical does not imply that the algebra is bispherical. 
\begin{example}
Let $A=KQ/I$ with 
$Q=$
% https://q.uiver.app/#q=WzAsMixbMCwwLCIxIl0sWzEsMCwiMiJdLFswLDEsImEiXSxbMSwwLCJiIiwwLHsib2Zmc2V0IjotM31dXQ==
\[\begin{tikzcd}
	1 & 2
	\arrow["a", shift left=1, from=1-1, to=1-2]
	\arrow["b", shift left=1, from=1-2, to=1-1]
\end{tikzcd}\]
$I=\langle a b \rangle.$
Then $A$ has global dimension 2 and all simple $A$-modules are spherical, but $A$ is not spherical since we know that a spherical Nakayama algebra has an acyclic quiver.
\end{example}
It would be interesting to study and classify algebras with all simple modules being spherical.

\section*{Acknowledgements}
This research started at the Junior Research Retreat in Bad Boll 2025 which was supported by the Hausdorff School for Mathematics (HSM) and funded by the Deutsche Forschungsgemeinschaft (DFG, German Research Foundation) under Germany's Excellence Strategy -- EXC-2047/2 -- 390685813.
This project benefited from the use of the GAP-package \cite{QPA}.
The second author is supported by JSPS KAKENHI Grant Numbers JP22KJ0737/JP25K17233, and is supported by Kavli Institute for the Physics and Mathematics of the Universe (WPI), The University of Tokyo, as an affiliate member.

\end{document}